\documentclass[a4paper,10pt]{article}

\usepackage{amsmath}
\usepackage{amsthm}
\usepackage{amsfonts}
\usepackage{amssymb}
\usepackage{graphicx}
\usepackage{dsfont}
\newcommand{\Prb}{\mathbb{P}}

\newcommand{\Or}{\textrm{O}}
\newcommand{\dist}{\textrm{d}}

\newenvironment{proofof}[1]{\textit{Proof of #1.}}{}

\newcommand{\Exp}{\textrm{exp}}
\newcommand{\diam}{\textrm{diam}}
\newcommand{\Max}[1]{\underset{#1}{\text{max}}}

\newtheorem{thm}{Theorem}

\newtheorem{definition}{Definition}
\newtheorem*{rem*}{Remark}
\newtheorem{prop}[thm]{Proposition}
\newtheorem{cor}[thm]{Corollary}

\newtheorem{conjecture}{Conjecture}
\newtheorem{lem}[thm]{Lemma}

\newenvironment{thm-hand}[1]	{\vspace{3mm}\noindent\textbf{Theorem #1}\begin{itshape}}
				{\end{itshape}\vspace{1mm}}

\newcounter{partcount}
\newenvironment{Parts}	{
				\begin{list}
				{\textbf{Part \arabic{partcount}}}
				{
					\usecounter{partcount}
					\setlength\leftmargin{0cm}
					\setlength\itemindent{0.2cm}
				}
			}{
				\end{list}
			}

\newcounter{casecount}
\newenvironment{Cases}	{
				\begin{list}
				{\textbf{Case \arabic{casecount}}}
				{
					\usecounter{casecount}
					\setlength\leftmargin{0cm}
					\setlength\itemindent{0.2cm}
				}
			}{
				\end{list}
			}

\title{A strict undirected model for the $k$-nearest neighbour graph}
\author{Neville Ball\footnote{n.ball@qmul.ac.uk, Queen Mary University of London, London, United Kingdom}}

\begin{document}
\maketitle



\begin{abstract}
Let $G=G_{n,k}$ denote the graph formed by placing points in a square of area $n$ according to a Poisson process of density 1 and joining each pair of points which are both $k$ nearest neighbours of each other. Then $G_{n,k}$ can be used as a model for wireless networks, and has some advantages in terms of applications over the two previous $k$-nearest neighbour models studied by Balister, Bollob\'{a}s, Sarkar and Walters, who proved good bounds on the connectivity models thresholds for both. However their proofs do not extend straightforwardly to this new model, since it is now possible for edges in different components of $G$ to cross. We get around these problems by proving that near the connectivity threshold, edges will not cross with high probability, and then prove that $G$ will be connected with high probability if $k>0.9684\log n$, which improves a bound for one of the models studied by Balister, Bollob\'{a}s, Sarkar and Walters too.
\end{abstract}



\section{Introduction}
Let $G_{n,k}$ be the graph formed by placing points in $S_{n}$, a $\sqrt{n}\times\sqrt{n}$ square, according to a Poisson process of density $1$ and connecting two points if they are both $k$-nearest neighbours of each other (i.e. one of the $k$-nearest points in $S_{n}$). We will refer to this as the strict undirected model. A natural question, especially when considering this as a model for a wireless network, is: Asymptotically, how large does $k$ have to be in order to ensure that $G_{n,k}$ is connected?

We cannot ensure with certainty that the resulting graph will be connected; there will always be a chance that a local configuration will occur that produces multiple components, but we can ask: what value of $k$ ensures that the probability of the graph being connected tends to one? Indeed we say that $G_{n,k}$ has a property $\Pi$ \emph{with high probability} if $\Prb(G_{n,k}\textrm{ has }\Pi)\rightarrow 1$ as $n\rightarrow\infty$. So we seek to answer the question: What $k=k(n)$ ensures that $G_{n,k}$ is connected with high probability?

Different variations of this problem have been studied previously, using different connection rules. Gilbert [\ref{Gil}] first introduced a model in which every point was joined to every other point within some fixed distance, $R$ (the Gilbert model). Equivalently, this can be viewed as joining each point, $x$, to every point within the circle of area $\pi R^{2}$ centred on $x$. Penrose proved in [\ref{Pen}], that if $\pi R^{2}\geq (1+o(1))\log n$ (so that on average each point is joined to at least $\log n$ other points), then the resulting graph is connected with high probability, whereas if $\pi R^{2}\leq (1+o(1))\log n$, then the resulting graph is disconnected with high probability.

Xue and Kumar [\ref{XandK}] studied the model in which two points are connected if either is the $k$-nearest neighbour of the other (we will denote this graph $G'_{n,k}$), and proved that the threshold for this model is $\Theta (\log n)$. Balister, Bollob\'{a}s, Sarkar and Walters [\ref{MW}] considerably improved their bounds (they showed that if $k<0.3043\log n$ then $G'_{n,k}$ is disconnected whp, while if $k>0.5139\log n$ then $G'_{n,k}$ is connected whp). In the same paper, Balister, Bollob\'{a}s, Sarkar and Walters also examined a directed version of the problem where a vertex sends out an out edge to all of its $k$ nearest neighbours, and again showed that the connectivity threshhold is $\Theta (\log n)$ obtaining upper and lower bounds of $0.7209\log n$ and $0.9967\log n$ respectively.

It has been pointed out that for practical uses (e.g. for wireless networks), it would be better to use a different connection rule, namely to connect two points only if they are both $k$ nearest neighbours of each other. This model has two advantages in terms of wireless networks: It ensures that no vertex will have too high a degree, and thus be swamped, as could happen with either of the previous models. It also ensures we can always receive an acknowledgement of any information sent at each step, which may not be the case in the directed model.

The edges in our new model are exactly the edges in the directed model which are bidirectional, and so any lower bound proved for the directed model will also be a lower bound for the strict undirected model. Thus, from Balister, Bollob\'{a}s, Sarkar and Walters [\ref{MW}] we know that if $k<0.7209\log n$ then $G_{n,k}$ is disconnected with high probability. It can be shown using a tessellation argument and properties of the Poisson process, that the connectivity threshold in this model is again $\Theta (\log n)$ (e.g. see the introduction of [\ref{MW}]), and so our task is to produce a good constant, $c$, for the upper bound such that if $k>c\log n$ then $G_{n,k}$ is connected with high probability. In particular we will show that some $c<1$ will do, to show that a conjecture of Xue and Kumar made for the original undirected model [\ref{XandK}] (and which is true for the Gilbert model) does not hold for this model. The method used in [\ref{MW}] for both of the previous models was to show first that for any $c'>0$, if $k>c'\log n$ then there could be only one `large' component of $G_{n,k}$ with high probability. This allowed them to concentrate on `small' components, and so gain their bounds.

We wish to do the same, however our model has some extra complications. One key property used in the proofs that there is only one large component was that edges in different components of $G$ cannot cross, but that is not the case in the strict undirected model. Indeed, Figure~\ref{FigCrossing} shows the outline of a construction in which the edges of two different components do cross.

\begin{figure}[h]
\centering
\includegraphics[height=80mm]{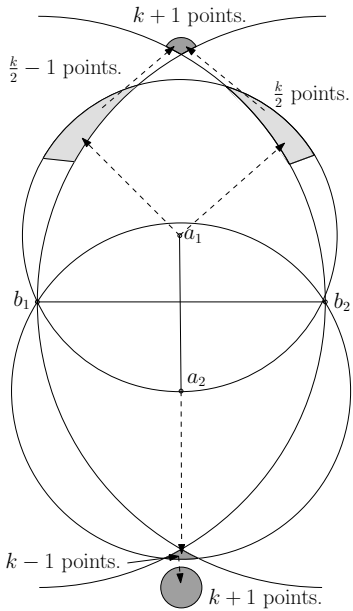}
\caption{If each of the shaded regions has the number of points shown, and there are no other points nearby, then $a_{1}a_{2}$ and $b_{1}b_{2}$ would be edges of $G_{n,k}$, but $a_{1}$ and $a_{2}$ would be in a different component from $b_{1}$ and $b_{2}$ (Here dashed arrows indicate directed out edges between regions).}
\label{FigCrossing}
\end{figure}

Luckily, the set-up required for edges of different components to cross is fairly restrictive, and we are able to show:

\begin{thm}\label{nocrossing}
If $k=c\log n$, then, for $c>0.7102$ (and in particular below the connectivity threshold), no two edges in different components inside $G$ will cross with high probability.
\end{thm}

\begin{rem*}
Officially this should read ``\emph{If $k=\lceil\log n\rceil$, then...},'' however, since we are considering the limit as $n$ tends to infinity, this makes no difference, and so for ease of notation we leave the ceiling notation out here, and for the rest of the paper.
\end{rem*}

There are further complications in proving good upper bounds on the connectivity threshold: In both of the previous models it was always the case that if there was no edge from a point $x$ to a point $y$, then there must be at least $k$ points closer to $x$ than $y$ is, whereas in our model we may only conclude that one or the other has $k$ nearer neighbours. For this reason we have to handle the case of small components differently too. We are able to show:

\begin{thm}\label{TightBoundThm}
If $k=c\log n$ and $c>0.9684$, then $G$ is connected with high probability.
\end{thm}

We first introduce some basic definitions and notation that will be used throughout the paper.

\section{Notation and Preliminaries}

\begin{definition}
Given a point $a\in G_{n,k} = G$, we write $\Gamma^{+}(a)$ for the set of the $k$-nearest neighbours of $a$ and define this to be the \emph{out neighbourhood of $a$}. We define the $k$-nearest neighbour disk of $a$, denoted $D^{k}(a)$, to be the smallest disk centred on $a$ that contains $\Gamma^{+}(a)$. 

We will often say that that a point $x$ has an \emph{out edge} to a point $y$ (or that $\overrightarrow{xy}$ is an out edge) to mean that $y\in \Gamma^{+}(x)$. Note that $xy$ is an edge in $G$ if and only if both $\overrightarrow{xy}$ and $\overrightarrow{yx}$ are out edges.
Correspondingly we say that $x$ has an \emph{in edge} from $y$ if $\overrightarrow{yx}$ is an out edge.
\end{definition}

We will use the following notational conventions:
\begin{itemize}
\item We write $D_{a}(r)$ for the disk of radius $r$ centred on $a$.
\item We will use capital letters to represent sets (e.g. a region of the plane, or a component), and lower case letters for points in the plane (however if $a$ and $b$ are points, we will write $ab$ for the edge (straight line segment) from $a$ to $b$).
\item For two sets $A$ and $B$, we write $\dist(A,B)$ for the minimum distance from any point in $A$ to any point in $B$. For a point $x$ and a region $B$ we write $\dist(x,B)=\dist(\{x\},B)$.
\item For a set $A$, we write $\partial A$ for the boundary of the closure of $A$.
\item Given a region $A$, we write $\#A$ for the number of points of $G$ in $A$, and $|A|$ for the area of $A$. We write $\Vert ab\Vert$ for the length of the edge $ab$.
\item We will refer to the vertices of $G$ as \emph{points} (i.e. points of our Poisson process), and a single element of $S_{n}$ as a \emph{location}.
\item We will often introduce Cartesian co-ordinates onto $S_{n}$ (with scaling), and when this is the case, we will write $p^{(x)}$ and $p^{(y)}$ for the $x$ and $y$ co-ordinates of any point/location~$p$.
\end{itemize}

At times we will refer to specific points and regions of $G$ and $S_{n}$, especially in the proof that edges of different components cannot cross (Section~\ref{NoCrossSection}), and so to help keep things easy to follow, a list of definitions and notations is included in Appendix~\ref{DefApp}.

\section{Edges of different components cannot cross, and there can only be one large component}

The eventual aim of this section will be to show that if $c=0.7102$ and $k>c\log n$, then with high probability there will only be one large component. We will achieve this by bounding the minimal distance between two edges in different components of $G$. As a first step we establish a lower bound on the distance of a point of $G$ and an edge in a different component.

\subsection{Preliminaries - An edge of one component cannot be too close to a vertex in another component}

To prove a bound on the distance between a point of $G$ and an edge in a different component, we first state the following result of Balister, Bollob\'{a}s, Sarkar and Walters [\ref{MW}] that bounds how close points in different components of $G$ can be. This lemma was proved for the original undirected model, but the proof uses properties of the Poisson process only. Namely, they showed that, given a point $x$, for any point $y$ that is close enough to $x$ we will have $\Prb(\overrightarrow{xy} \textrm{ not an out edge})=O(n^{1-\varepsilon})$, and thus that with high probability all points close enough together have out edges to each other. Since this implies $\overrightarrow{xy}$ and $\overrightarrow{yx}$ are both out edges for $x$ and $y$ close enough together, it also shows that $xy$ would be an edge in our model.

\begin{lem}\label{edgelengths}
Fix $c>0$, and set;
\[c_{-}=ce^{-1-1/c}\textrm{	and	}c_{+}=4e(1+c)\]
If $r$ and $R$ are such that $\pi r^{2}=c_{-}\log n$ and $\pi R^{2}=c_{+}\log n$, then whp every vertex in $G_{n,k}$ is joined to every vertex within distance $r$, and every vertex has at least $k+1$ other vertices within a distance $R$, and so in particular is not joined to any vertex more than a distance $R$ away.
\end{lem}

The next lemma will be used repeatedly, and is a result about how points can be connected in our graph. It states that the longest edge (in $G$) out of any point, $x$, is at most twice the shortest non-edge involving $x$, or, equivalently, that the region containing the neighbourhood of $x$ (in $G$) is at most a factor of two off being circular. This is certainly not the case in either of the two previous models.

\begin{lem}\label{D1/2}
Let $x$ and $y$ be two points of $G$ such that $D^{k}(x)\subset D^{k}(y)$, then $x$ is joined to $y$, and $\Gamma^{+}(x)\cup\{x\}=\Gamma^{+}(y)\cup\{y\}$. In particular, if $xy$ is an edge of $G$ then $x$ must be joined to every point inside $D_{x}(\Vert xy\Vert/2)$.
\end{lem}

\begin{proof}
Since $D^{k}(x)\subset D^{k}(y)$, the $k$ nearest neighbours of $y$ must all lie inside $D^{k}(x)$. If $y\notin D^{k}(x)$, then $D^{k}(y)$ contains $k+2$ points ($k+1$ in $D^{k}(x)$), which is impossible. Thus $xy$ is an edge of $G$ and the set of points (excluding $x$ and $y$) in $D^{k}(x)$ is precisely the same as those in $D^{k}(y)$.

To prove the last part, suppose that $z$ is a point in $D_{x}(\Vert xy\Vert/2)$. Then $\overrightarrow{xz}$ must be an out edge, since $\Vert xz\Vert<\Vert xy\Vert$. Now, if $\overrightarrow{zx}$ is not an out edge then $x\notin D^{k}(z)$, but $z\in D_{x}(\Vert xy\Vert/2)$, and so $D^{k}(z)\subset D_{x}(\Vert xy\Vert)\subset D^{k}(x)$. But this implies $xz\in G$ by the above.
\end{proof}

We will now show that there is an absolute minimum distance between a point and a edge from a different component. As the main step to doing so, (and for most of the rest of this subsection) we show that there is a relative minimum distance between an edge of $G$ and the distance of a point from a different component to that edge (as a function of the length of the edge). This result will be used both as the main part of that result of an absolute minimum distance, and later as part of the proof that with high probability edges in different components cannot cross. To this end we prove a fairly strong result and introduce a lot of the notation and set-up which we will meet again when proving that edges will not cross with high probability.

\begin{lem}\label{farapart1} Suppose $b_{1}$ and $b_{2}$ are in a component $X$, with $b_{1}b_{2}\in G$, $\Vert b_{1}b_{2}\Vert=\rho$ and $a\notin X$, then:
\begin{align}
\textrm{d}(a,b_{1}b_{2}) & \geq\frac{1}{4\sqrt{6}} \rho > 0.102\rho
\end{align}
\end{lem}
\begin{proof}
Suppose $a$, $b_{1}$ and $b_{2}$ are as above. We rescale and introduce Cartesian co-ordinates, fixing $b_{1}$ at $(0,0)$ and $b_{2}$ at $(1,0)$. Without loss of generality, $a^{(y)}\geq 0$ and $a^{(x)}\leq\frac{1}{2}$. We need to show that $\textrm{d}(a,b_{1}b_{2})\geq\frac{1}{4\sqrt{6}}$. We write $B_{i}$ for $D_{b_{i}}(1)$, and note that $B_{i}\subset D^{k}(b_{i})$ (as the edge $b_{1}b_{2}\in G$). We may assume that $a\in B_{1}$, since otherwise $\textrm{d}(a,b_{1}b_{2})\geq\frac{\sqrt{3}}{2}$ (as $a_{1}^{(x)}\leq 1/2$).

Since $a$ is not joined to either $b_{i}$, Lemma~\ref{D1/2} tells us that:
\begin{align}
a & \notin D_{b_{1}}(1/2)\cup D_{b_{2}}(1/2)\label{aoutside}
\end{align}

If $a^{(x)}<0$, then, using (\ref{aoutside}), $\textrm{d}(a,b_{1}b_{2})>1/2$. Thus we may assume $0< a^{(x)}\leq 1/2$, so that we have $\textrm{d}(a,b_{1}b_{2})=a^{(y)}$.

Let $w$ be the location $(\frac{1}{2},\frac{1}{2\sqrt{3}})$, and let $T$ be the triangle with vertices $b_{1}$, $b_{2}$ and $w$ (See figure \ref{FigTandT2}).

Note that $b_{1}\widehat{b_{2}}w=b_{2}\widehat{b_{1}}w=\frac{\pi}{6}$, and so $T$ intersects $D_{b_{1}}(1/2)$ and $D_{b_{2}}(1/2)$ at $(\frac{\sqrt{3}}{4},\frac{1}{4})$ and $(1-\frac{\sqrt{3}}{4},\frac{1}{4})$ respectively. In particular, (\ref{aoutside}) tells that if $a\notin T$ then $\textrm{d}(a,b_{1}b_{2})\geq \frac{1}{4}$.

Thus we may assume that $\overrightarrow{b_{1}a}$ and $\overrightarrow{b_{2}a}$ are out edges, and that:
\begin{align}
a & \in S = \left(T\cap\{p:p^{(x)}<\frac{1}{2}\}\right)\setminus D_{b_{1}}(1/2)\label{ainside}
\end{align}
See Figure~\ref{FigTandT2}.

\begin{figure}[h]
\centering
\includegraphics[height=60mm]{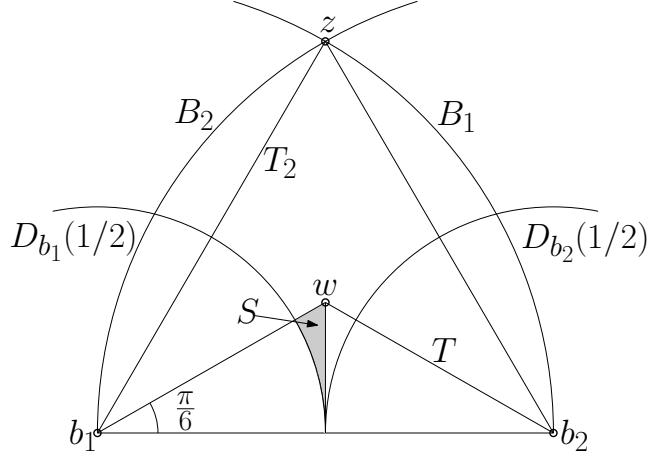}
\caption{The region we are considering for $a$, shown with $T$ and $T_{2}$.}
\label{FigTandT2}
\end{figure}

Define $r=\Vert ab_{1}\Vert$, and write $A$ for the disk $D_{a}(r)$, so that $\Gamma^{+}(a)\subset D^{k}(a)\subset A$. Since $a\in S$, we have:
\begin{align}
r & \leq \Vert b_{1}w\Vert=\frac{1}{\sqrt{3}}\label{rsmall}
\end{align}

Let $z$ be the location $(\frac{1}{2},\frac{\sqrt{3}}{2})$. Note that $b_{1}$, $b_{2}$ and $z$ form an equilateral triangle~$T_{2}$ that contains $T$ (See figure \ref{FigTandT2}). Note that for any point in $T_{2}$ (and so, in particular, for every point in $S$), $z$ is the closest point on $\partial (B_{1}\cup B_{2})$. Thus:
\begin{align}
\dist(a,\partial (B_{1}\cup B_{2})) & = \Vert az\Vert \geq \Vert wz\Vert = \frac{1}{\sqrt{3}}\label{daBig}
\end{align}
Thus, putting (\ref{rsmall}) and (\ref{daBig})together, we have:
\begin{align}
D^{k}(a) & \subset A \subset B_{1}\cup B_{2} \label{DkaInside}
\end{align}

Now, Lemma~\ref{D1/2} tells us that we cannot have $\Gamma^{+}(a)\subset B_{i}$ for either $i$, and so $\Gamma^{+}(a)$ (and thus $A$) must contain points in both $B_{1}\setminus B_{2}$ and $B_{2}\setminus B_{1}$. We consider a point $p\in \Gamma^{+}(a)\cap (B_{2}\setminus B_{1})$. By definition, both $b_{2}$ and $a$ must have an out edge to $p$, and thus, since $a$ and $b_{2}$ are in different components, one of the following must hold:
\begin{enumerate}
\item $p$ has no out edge to $a$.\label{nopa}
\item $p$ has no out edge to $b_{2}$.\label{nopb}
\end{enumerate}

We will show that if $a$ is too close to $b_{1}b_{2}$, then $A$ (and so $\Gamma^{+}(a)$) cannot contain a suitable point with either of these conditions holding. In particular, writing $E$ for the ellipse $\{p:\Vert ap\Vert + \Vert b_{2}p\Vert\leq1\}$, we show that if $a$ is too close to $b_{1}b_{2}$ then $R:=A\cap(B_{2}\setminus B_{1})\subset E\cap D_{b_{2}}(1/2)$, and that no point in $E\cap D_{b_{2}}(1/2)$ can satisfy either of the above conditions.

\begin{lem}\label{EllipseLemma}
If $p\in E$ then $\overrightarrow{pa}$ is an out edges. In particular, if $p\in E\cap D_{b_{2}}(1/2)$, then both $\overrightarrow{pa}$ and $\overrightarrow{pb_{2}}$ are out edges.
\end{lem}
\begin{proof}
Suppose that $p\in E$ and $\overrightarrow{pa}$ is not an out edge. We must have $a\notin D^{k}(p)$, and so $D^{k}(p)\subset B_{2}\subset D^{k}(b_{2})$ by the definition of $E$. Thus lemma \ref{D1/2} tells us that $\Gamma^{+}(p)\cup\{p\}=\Gamma^{+}(b_{2})\cup\{b_{2}\}$. But $a\in \Gamma^{+}(b_{2})$, and so $a\in\Gamma^{+}(p)$, and we have a contradiction.

The second part follows by applying Lemma~\ref{D1/2}.
\end{proof}

We now identify a location, $q$, which is quite high up on $\partial B_{1}$ and must be inside $E\cap D_{b_{2}}(1/2)$. Lemma~\ref{EllipseLemma} tells us that $R$ must contain a point further round $\partial B_{1}$ than $q$, or else $a$ and $b_{2}$ are in the same component. This will force $a$ itself to not be too close to $b_{1}b_{2}$.

\begin{lem}
Let $q=(\frac{11}{12},\frac{\sqrt{23}}{12})$. Then, so long as $a\in S$, $q\in E\cap D_{b_{2}}(1/2)$.
\end{lem}
\begin{proof}
We have that $\Vert qb_{2}\Vert=\sqrt{(\frac{1}{12})^{2}+(\frac{\sqrt{23}}{12})^{2}}=\frac{1}{\sqrt{6}}<\frac{1}{2}$. Thus $q\in D_{b_{2}}(1/2)$, and moreover $q\in E$ if and only if $a\in D_{q}(1-\frac{1}{\sqrt{6}})$.

Since $S$ is contained within its complex hull, we will have $a\in D_{q}(1-\frac{1}{\sqrt{6}})$ so long as the corners of $S$ are contained within $D_{q}(1-\frac{1}{\sqrt{6}})$. Now, $S$ has three corners: $(\frac{1}{2},0)$, $(\frac{\sqrt{3}}{4},\frac{1}{4})$ and $(\frac{1}{2},\frac{1}{2\sqrt{3}})$, and by some simple calculations:
\begin{align*}
\textrm{d}(q,(\frac{1}{2},\frac{1}{2\sqrt{3}})) < \textrm{d}(q,(\frac{1}{2},0)) & < 1-\frac{1}{\sqrt{6}}
\end{align*}
And:
\begin{align*}
\textrm{d}(q,(\sqrt{3}/4,1/4)) & < 1-\frac{1}{\sqrt{6}}
\end{align*}
Thus all these locations are inside $D_{q}(1-\frac{1}{\sqrt{6}})$, and we are done.
\end{proof}

Note that $\Vert qb_{1}\Vert=1$ and so $q\in\partial B_{1}$. Now, $R$ must have its location furthest from $b_{2}$ on $\partial B_{1}$ (since $b_{2}\in \partial B_{1}$ and $a\in B_{1}$), and so if $R$ contains any location outside of $E\cap D_{b_{2}}(1/2)$ it must contain a location further up $\partial B_{1}$ than $q$.

Since $R$ is symmetric about the line through $a$ and $b_{1}$, $R$ could only contain a location above $q$ if $a$ is above the bisector of angle $q\widehat{b_{1}}b_{2}$ (denote this line $L$). Since we are assuming $a\in S$, we must have that $a^{(y)}$ (and so $\textrm{d}(a,b_{1}b_{2})$) is at least the second co-ordinate of the intersection between $\partial D_{b_{1}}(1/2)$ and $L$.

Writing $2\theta$ for $q\widehat{b_{1}}b_{2}$, we have that:
\begin{align}
\sin^{2} \theta & = \frac{1-\cos2\theta}{2}= \left(1-\frac{11/12}{\sqrt{(11/12)^{2}+(\sqrt{23}/12)^{2}}}\right)/2=\frac{1}{24}
\end{align}
Now, $a$ must be above the location which is $1/2$ along the line $L$ from $b_{1}$ (since $a\notin D_{b_{1}}(1/2)$). Thus:
\begin{align}
a^{(y)} & \geq \frac{1}{2}\sin\theta = \frac{1}{2}\frac{1}{\sqrt{24}} = \frac{1}{4\sqrt{6}}
\end{align}
\end{proof}

We want to bound the distance between a point and an edge in a different component independent of the length of the edge. We do this by applying Lemma~\ref{edgelengths} if the edge is short, and Lemma~\ref{farapart1} if the edge is long:

\begin{cor}\label{farapart2}
With $r$ as defined in Lemma~\ref{edgelengths}, we have that if $b_{1}$ and $b_{2}$ are in a component $X$ with $b_{1}b_{2}\in G$, and $a\notin X$, then;
\begin{align}
\dist(a,b_{1}b_{2}) & > \frac{r}{5}
\end{align}
\end{cor}

\begin{proof}
Suppose $b_{1}$, $b_{2}$ and $a$ are as above and let $\Vert b_{1} b_{2} \Vert=\rho$.

If $\rho\leq \frac{4\sqrt{6}}{5}r$: We may assume $\Vert ab_{1}\Vert\leq \Vert ab_{2}\Vert$. Then the perpendicular projection of $a$ onto $b_{1}b_{2}$ is at most $\rho/2$ from $b_{1}$. Thus, since $ab_{1}$ is not an edge of $G$, Lemma~\ref{edgelengths} tells us that $\Vert ab_{1}\Vert\geq r$ and so:
\begin{align}
\dist(a,b_{1}b_{2}) & \geq \sqrt{r^{2}-(\rho/2)^{2}} \geq \sqrt{r^{2}-(\frac{2\sqrt{6}}{5}r)^{2}}=\frac{r}{5}
\end{align}

If $\rho\geq\frac{4\sqrt{6}}{5}r$: By Lemma~\ref{farapart1} we have that:
\begin{align}
\dist(a,b_{1}b_{2}) & \geq \frac{1}{4\sqrt{6}}\rho\geq \frac{r}{5}
\end{align}
\end{proof}

\begin{rem*}
Lemma~\ref{farapart1} can be improved, with substantial extra work, to show the distance between $a$ and $b_{1}b_{2}$ is at least $0.1934\rho$, which is best possible.
\end{rem*}

\subsection{Proof of Theorem~\ref{nocrossing} - Edges in different components cannot cross}\label{NoCrossSection}

In this section we will show:

\begin{thm-hand}{\ref{nocrossing}}
If $k=c\log n$, then, for $c>0.7102$, no two edges in different components inside $G$ will cross with high probability.
\end{thm-hand}

The value $c=0.7102$ is strictly less than the current lower bound on the connectivity constant (i.e. $c=0.7209$), and so edges in different components stop crossing before everything is connected.

The proof of Theorem~\ref{nocrossing} will split into three main parts. In the first we prove that for two such edges to cross, there must be a fairly specific set-up of points, more precisely it must look similar to the construction in Figure~\ref{FigCrossing}. In the second section we show that we can define two regions within this set-up, one of which has high density (containing at least $k$ points and denoted $H$), and the other of which is empty (and denoted $L$). In the third section we bound the relative sizes of these two regions, and so achieve a bound on the likelihood of such a set-up occurring by using the following result of Balister, Bollob\'{a}s, Sarkar and Walters [\ref{MW}], proved using simple properties of the Poisson process:

\begin{lem}\label{Full-Empty}\mbox{}
If $X$ and $Y$ are two regions of the plain, then:
\begin{align}
\Prb(\#X\geq k\text{ and }\#Y=0) & \leq \left(\frac{|X|}{|X|+|Y|}\right)^{k}\notag
\end{align}
\end{lem}

It is worth remarking that there will exist a constant $c'$ such that if $k<c'\log n$ then with high probability we would have edges in different components crossing: We have a construction where we do have two edges in different components crossing (see Figure~\ref{FigCrossing} in the introduction). Now, the construction has 5 dense regions, which we denote $H_{i}$ ($i=1,\ldots,5$), each of which contains $m_{i}$ points, ($\sum_{i}m_{i}=4k$) and a large empty regions, which we will denote $L$. If we have a region of the right shape with an area equal to the number of points in the construction (namely $4k$), then, writing $p_{n}$ for the probability of the construction occurring in that region, we have:
\begin{align}
p_{n}&>\prod_{i}^{5}\left(\frac{|H_{i}|}{|L\cup H_{i}|}\right)^{m_{i}}\notag\\
&>\underset{|H_{i}|}{\text{min}}\left(\frac{|H_{i}|}{|L\cup H_{i}|}\right)^{4k}\notag\\
&=n^{4c'\underset{|H_{i}|}{\text{min}}\log\frac{|H_{i}|}{|L\cup H_{i}|}}\label{ExpIntro}
\end{align}
when $k=c'\log n$. Now, by taking $c'$ to be small enough, we can make the exponent of (\ref{ExpIntro}) arbitrarily close to $0$, and so the probability of such a set-up occurring can be $\Or(n^{-\varepsilon})$ for any $\varepsilon>0$. Since the region had an area of $\Or(\log n)$, we can fit $\Or(n/\log n)$ disjoint copies into $S_{n}$. Thus if we partition $S_{n}$ into $\Or(n/\log n)$ regions in each of which the set-up could occur, it will occur in some of them with high probability, and so $G$ will contain components with crossing edges with high probability.

\subsubsection{The set-up of the points}

To prove the result, we need to refer to several specific regions and locations within $S_{n}$, and so to make it easier to follow, all definitions and notation within this section are collated in the order that they appear in Appendix~\ref{DefApp}, in addition to being defined inside this section.

\begin{definition}
We say that the ordered set of points: $(a_{1}, a_{2}, b_{1}, b_{2})$ forms a \emph{crossing pair} if:
\begin{itemize}
\item The straight line segments $a_{1}a_{2}$ and $b_{1}b_{2}$ intersect and are both edges of the graph $G$,
\item the points $a_{1}$ and $a_{2}$ are in a different component from $b_{1}$ and $b_{2}$,
\item $\Vert a_{1}a_{2}\Vert\leq\Vert b_{1}b_{2}\Vert$, $\Vert a_{1}b_{1}\Vert\leq \Vert a_{1}b_{2}\Vert$ and $\textrm{d}(a_{1},b_{1}b_{2})\leq\textrm{d}(a_{2},b_{1}b_{2})$.
\end{itemize}
\end{definition}

Note that any four points that meet the first two conditions must also meet the third under a suitable identification of points, so that if two edges from different components cross then some four points must form a crossing pair.

We will use this definition of crossing pairs to determine exactly how a set-up with two edges from different components crossing must look. Given a crossing pair, we introduce Cartesian co-ordinates and rescale exactly as in Lemma~\ref{farapart1} throughout this section (i.e. setting $b_{1}=(0,0)$, $b_{2}=(1,0)$, $a_{1}^{(x)}\leq 1/2$, $a_{1}^{(y)}\geq 0$ and $a_{2}^{(y)}\leq 0$). We now introduce some definitions of regions (dependent on $a_{1}$, $a_{2}$, $b_{1}$ and $b_{2}$), which we will use to pin point where these points can lie in relation to each other:

\begin{definition}
Let $r_{i}=\textrm{min}\{\Vert a_{i}b_{1}\Vert,\Vert a_{i}b_{2}\Vert\}$ (so that $r_{1}=\Vert a_{1}b_{1}\Vert$) and define $A_{i}=D_{a_{i}}(r_{i})$ and  $B_{i}=D_{b_{i}}(\Vert b_{1}b_{2}\Vert)=D_{b_{i}}(1)$ (See Figure~\ref{FigA1A2B1B2}).

\begin{figure}[h]
\centering
\includegraphics[height=50mm]{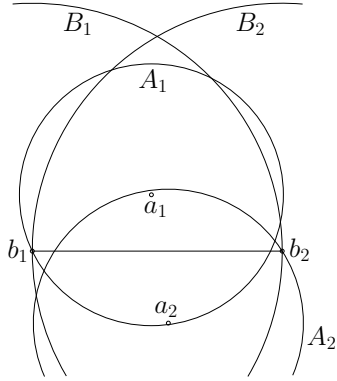}
\caption{The regions $A_{1}$, $A_{2}$, $B_{1}$ and $B_{2}$.}
\label{FigA1A2B1B2}
\end{figure}
\end{definition}

\begin{definition}
We write $T$ for the isosceles triangle with vertices $b_{1}$, $b_{2}$ and $w$ where $w=(\frac{1}{2},\frac{1}{2\sqrt{3}})$, and $S_{1}$ for the region $\left(T\cap\{q:q^{(x)}\leq1/2\}\right)\setminus D_{b_{1}}(1/2)$ (This will turn out to be the region which can contain $a_{1}$. See Figure~\ref{RegionForA1}).

\begin{figure}[h]
\centering
\includegraphics[height=60mm]{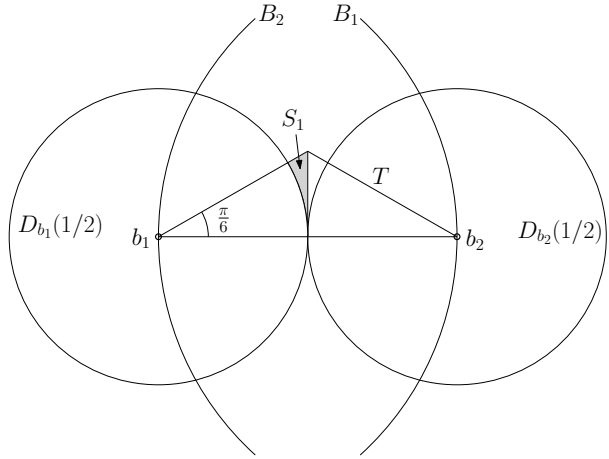}
\caption{The shaded region is the region $S_{1}$ (which can contain $a_{1}$).}
\label{RegionForA1}
\end{figure}
\end{definition}

\begin{definition}
We write $T_{2}$ for the equilateral triangle with vertices $b_{1}$, $b_{2}$ and $z$, where $z=(\frac{1}{2},-\frac{\sqrt{3}}{2})$, and $S_{2}$ for the region $T_{2}\cap A_{1}\cap\{x:x\widehat{b_{1}}b_{2}> \pi/6\textrm{ and }x\widehat{b_{2}}b_{1}> \pi/6\}$ (This will turn out to be the region that can contain $a_{2}$. See Figure~\ref{RegionForA2}).

\begin{figure}[h]
\centering
\includegraphics[height=60mm]{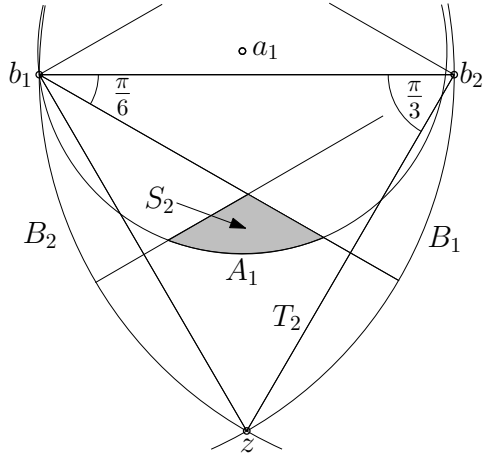}
\caption{The shaded region is the region $S_{2}$ (which can contain $a_{2}$).}
\label{RegionForA2}
\end{figure}
\end{definition}

\begin{definition}
For any set $S$, we define $S^{+}$ to be the part of $S$ that lies above the $x$-axis (i.e. the line through $b_{1}$ and $b_{2}$), and $S^{-}$ to be the part of $S$ that lies below the $x$-axis.
\end{definition}

To show that $a_{1}\in S_{1}$ and $a_{2}\in S_{2}$, (as well as later) we will need the following generalisation of Lemma~\ref{D1/2} to pairs of points:

\begin{lem}\label{IntersectUnion}
Suppose $w$, $x$, $y$ and $z$ are any four points such that:
\begin{enumerate}
\item $D^{k}(w)\cup D^{k}(x)\subset D^{k}(y)\cup D^{k}(z)$,\label{CondUnion}
\item $D^{k}(w)\cap D^{k}(x)\subset D^{k}(y)\cap D^{k}(z)$.\label{CondInter}
\end{enumerate}
Then at least one of $wy$, $wz$, $xy$ and $xz$ is an edge of $G$.
\end{lem}
\begin{proof}
Let $\#(D^{k}(w)\cap D^{k}(x))=m$ and $\#(D^{k}(y)\cap D^{k}(z))=\mu$. Then, by condition~\ref{CondInter}, $m\leq\mu$. However, $\#(D^{k}(w)\cup D^{k}(x))=2k+2-m$ and $\#(D^{k}(y)\cup D^{k}(z))=2k+2-\mu$, and so so condition~\ref{CondUnion} implies $2k+2-m\leq2k+2-\mu$ and thus $m\geq\mu$. Putting these together, we must have $m=\mu$.

This tells us that $\#(D^{k}(w)\cup D^{k}(x))=\#(D^{k}(y)\cup D^{k}(z))$, and so, by condition~\ref{CondUnion}, we have $\Gamma^{+}(w)\cup\Gamma^{+}(x)\cup\{w,x\}=\Gamma^{+}(y)\cup\Gamma^{+}(z)\cup\{y,z\}$. In particular $w,x\in\Gamma^{+}(y)\cup\Gamma^{+}(z)$ and $y,z\in\Gamma^{+}(w)\cup\Gamma^{+}(x)$, and so each of $w$ and $x$ receives an out-edge from at least one of $y$ and $z$ and each of $y$ and $z$ receives an out-edge from at least one of $w$ and $x$. We may assume by symmetry that $\overrightarrow{wy}$ is an out-edge.

Now, if $wy$ were not an edge of $G$, then $\overrightarrow{zw}$ must be an out-edge (since one of $\overrightarrow{yw}$ and $\overrightarrow{zw}$ must be). Similarly, if $zw$ is not an edge of $G$ either, then $\overrightarrow{xz}$ must be an out edge. Continuing, we find that either one of $wy$, $wz$, $xy$ and $xz$ is an edge of $G$, or all of $\overrightarrow{wy}$, $\overrightarrow{zw}$, $\overrightarrow{xz}$, $\overrightarrow{yx}$ are out-edges, but none are in-edges. This would imply:\[\Vert wy\Vert<\Vert zw\Vert<\Vert xz\Vert<\Vert yx\Vert<\Vert wy\Vert,\]which is impossible.
\end{proof}

We now finish this sub-section by showing that $a_{1}\in S_{1}$ and $a_{2}\in S_{2}$, and proving some other basic facts about crossing pairs:

\begin{lem}\label{Properties}
Suppose $(a_{1}, a_{2}, b_{1}, b_{2})$ forms a crossing pair, then:
\begin{enumerate}
\item \label{a1a2Short}$a_{1}a_{2}$ must be the shortest edge in the convex quadrilateral $a_{1}a_{2}b_{1}b_{2}$,
\item \label{AsAndBs}we must have $0<a_{1}^{(x)},a_{2}^{(x)}<1$, and $B_{i}\subset D^{k}(b_{i})$ and $\Gamma^{+}(a_{i})\subset A_{i}$ for $i=1,2$,
\item \label{Positiona1}$a_{1}\in S_{1}$,
\item \label{T2Lemma}for any point $p\in T_{2}$ with $b_{1}$, $b_{2}\notin D^{k}(p)$, if either of $b_{1}\widehat{b_{2}}p\leq\pi/6$ or $b_{2}\widehat{b_{1}}p\leq\pi/6$ then $D^{k}(p)\subset B_{1}\cup B_{2}$,
\item \label{Positiona2}$a_{2}\in S_{2}$.
\end{enumerate}
\end{lem}

\begin{proof}
\begin{enumerate}
\item Since $a_{1}a_{2}$ and $b_{1}b_{2}$ intersect, the four points must form a convex quadrilateral with $a_{1}a_{2}$ and $b_{1}b_{2}$ as the diagonals.

Suppose $a_{1}b_{1}$ is shorter than $a_{1}a_{2}$ (and so also shorter than $b_{1}b_{2}$), then $a_{1}\in D^{k}(b_{1})$ as $b_{2}$ is, and $b_{1}\in D^{k}(a_{1})$ as $a_{2}$ is. Thus $a_{1}b_{1}$ is an edge in $G$, contradicting $(a_{1}, a_{2}, b_{1}, b_{2})$ being a crossing pair. Similarly, $a_{i}b_{j}$ cannot be shorter than both $a_{1}a_{2}$ for any $i$ and $j$.

\item We know that $b_{1}b_{2}\in G$, and thus $B_{i}\subset D^{k}(b_{i})$, and know already that $a_{1}^{(x)}\leq \frac{1}{2}$.

Suppose that $a_{1}^{(x)}\leq0$. Since $a_{1}a_{2}$ and $b_{1}b_{2}$ intersect, we must have $a_{2}^{(x)}>0$. But then $\Vert b_{1}a_{2}\Vert<\Vert a_{1}a_{2}\Vert$, contradicting part~\ref{a1a2Short}. Thus $a_{1}^{(x)}>0$. The same argument shows that  $a_{2}^{(x)}>0$ and $a_{2}^{(x)}<1$.

By the above, and using $\Vert a_{1}a_{2}\Vert\leq \Vert b_{1}b_{2}\Vert=1$ as well as $\dist(a_{1},b_{1}b_{2})\leq\dist(a_{2},b_{1}b_{2})$, we have that $0\leq a_{1}^{(y)}=\dist(a_{1},b_{1}b_{2})\leq\frac{1}{2}$. We also know that $0<a_{1}^{(x)}\leq \frac{1}{2}$, and so $\Vert a_{1}b_{1}\Vert\leq\frac{1}{\sqrt{2}}$, and in particular $a_{1}\in B_{1}$.

Thus $\overrightarrow{b_{1}a_{1}}$ is an out edge, and so $b_{1}\notin\Gamma^{+}(a_{1})$ as $a_{1}b_{1}$ is not an edge of $G$. This implies that $b_{2}\notin\Gamma^{+}(a_{1})$ as $a_{1}^{(x)}\leq\frac{1}{2}$. Thus $D^{k}(a_{1})\subset A_{1}$.

Since neither $b_{1}$ nor $b_{2}$ are in $A_{1}$ and $0<a_{1}^{(x)}\leq \frac{1}{2}$, we must have $(\partial A_{1})^{-}\subset B_{1}\cap B_{2}$. Thus $D^{k}(a_{1})^{-}\subset A^{-}\subset B_{1}\cap B_{2}$, and so $a_{2}\in B_{1}\cap B_{2}$ implying that $\overrightarrow{b_{1}a_{2}}$ and $\overrightarrow{b_{2}a_{2}}$ are both out edges. Thus neither $b_{1}$ nor $b_{2}$ are in $\Gamma^{+}(a_{2})$, so $D^{k}(a_{2})\subset A_{2}$.

\item We must have $2d(a_{1},b_{1}b_{2})\leq \Vert a_{1}a_{2}\Vert \leq \Vert a_{1}b_{1}\Vert$, since $0<a_{1}^{(x)},a_{2}^{(x)}<1$ and $a_{1}a_{2}$ is the shortest edge in our quadrilateral, and so in particular:
\begin{align*}
d(a_{1},b_{1}b_{2})\leq \frac{1}{2}\Vert a_{1}b_{1}\Vert
\end{align*}
Thus, using $\Vert a_{1}b_{1}\Vert\leq\Vert a_{1}b_{2}\Vert$:
\begin{align}
a_{1}\widehat{b_{2}}b_{1}\leq a_{1}\widehat{b_{1}}b_{2}\leq \sin^{-1}(\frac{1}{2})=\pi/6
\end{align}
This is exactly the region $T$, and since $a_{1}^{(x)}\leq1/2$ and $a_{1}\notin D_{b_{1}}(1/2)$ (by Lemma~\ref{D1/2}), we have:
\begin{align}
a_{1}\in \left(T\cap\{q:q^{(x)}\leq1/2\}\right)\setminus D_{b_{1}}(1/2) = S_{1}\notag
\end{align}

\item Let $p\in T_{2}$ be such that $b_{1}, b_{2}\notin D^{k}(p)$. Note that $z$ is the closest location to $p$ in $\partial (B_{1}\cup B_{2})$ (since $p\in T_{2}$), and so in particular $D_{p}(\Vert pz\Vert)\subset B_{1}\cup B_{2}$. Thus it suffices to show that $z\notin D^{k}(p)$.

If $b_{1}\widehat{p}b_{2}\leq\pi/6$, then $\Vert b_{1}p\Vert\leq \Vert pz\Vert$ since the line $\{q:b_{1}\widehat{b_{2}}q=\pi/6\}$ bisects $b_{1}\widehat{b_{2}}z$. Thus in particular, $z\notin D^{k}(p)$ since $b_{1}\notin D^{k}(p)$.

Similarly, if $b_{2}\widehat{p}b_{1}\leq\pi/6$ then $z\notin D^{k}(p)$.

\item Noting that the $a_{i}$ and $b_{i}$ fulfil condition~2 of Lemma~\ref{IntersectUnion} (with the identification, in the notation of Lemma~\ref{IntersectUnion}, of $a_{1}=w$, $a_{2}=x$, $b_{1}=y$ and $b_{2}=z$), and so, since the $a_{i}$ and $b_{i}$ are in different components, Lemma~\ref{IntersectUnion} implies that $A_{1}\cup A_{2}\not\subset B_{1}\cup B_{2}$. Thus at least one of $a_{1}$ and $a_{2}$ must be closer to a point outside of $B_{1}\cup B_{2}$ than it is to $b_{1}$ and $b_{2}$. This cannot be $a_{1}$ by parts~\ref{Positiona1} and \ref{T2Lemma}. Thus $a_{2}$ is closer to a point outside of $B_{1}\cup B_{2}$ than it is to $b_{1}$ or $b_{2}$.

Since $a_{1}a_{2}$ is the shortest edge in both triangles $a_{1}a_{2}b_{1}$ and $a_{1}a_{2}b_{2}$, we have $a_{1}\widehat{b_{i}}a_{2}\leq\pi/3$ for $i=1,2$, and so $a_{2}\in T_{2}$. Thus by part~\ref{T2Lemma}, $a_{2}\widehat{b_{1}}b_{2}> \pi/6$ and $a_{2}\widehat{b_{2}}b_{1}> \pi/6$. We also know that $a_{2}\in A_{1}$ as $a_{1}a_{2}\in G$, whence:
\begin{align}
a_{2}\in T_{2}\cap A_{1}\cap\{x:x\widehat{b_{1}}b_{2}> \pi/6\textrm{ and }x\widehat{b_{2}}b_{1}> \pi/6\}=S_{2}\notag
\end{align}
\end{enumerate}
\end{proof}

\subsubsection{The dense and empty regions}

We want to define our regions of high and low density, but first need some more basic regions that they will be built from. We define:
\begin{itemize}
\item $R_{i}$ to be $D^{k}(a_{1})\cap (B_{i}\setminus B_{j})$ where $i\neq j$,
\item $E_{i}$ to be the ellipse defined by the equation $\Vert a_{1}x\Vert+\Vert b_{i}x\Vert\leq1$ (This has its centre half way between $a_{1}$ and $b_{i}$, major axis running along the line $a_{1}b_{i}$ with radius $1/2$, and minor axis of radius $\frac{\sqrt{1-r_{i}^{2}}}{2}$),
\item $F_{i}$ to be the ellipse defined by the equation $\Vert a_{2}x\Vert+\Vert b_{i}x\Vert\leq1$,
\item $M$ to be $D^{k}(a_{1})\cap D^{k}(a_{2})$.
\end{itemize}

We can now define all our regions of high and low density (and will prove they are such shortly). All these regions are shown in Figure~\ref{FigHandL}. The empty regions are:
\begin{itemize}
\item $L_{1}=(D^{k}(a_{1})^{+}\cap E_{1}\cap D_{b_{1}}(1/2))\setminus M$
\item $L_{2}=(D^{k}(a_{1})^{+}\cap E_{2}\cap D_{b_{2}}(1/2))\setminus M$
\item $L_{3}=M^{+}\cap (D_{b_{1}}(1/2)\cup D_{b_{2}}(1/2))$
\item $L_{4}=T_{2}\cap D^{k}(a_{2})\cap \{x:x\widehat{b_{1}}b_{2}\leq \pi/6\textrm{ or }x\widehat{b_{2}}b_{1}\leq \pi/6\}$
\item $L_{5}=(D^{k}(a_{2})^{-}\cap F_{1}\cap D_{b_{1}}(1/2))\setminus T_{2}$
\item $L_{6}=(D^{k}(a_{2})^{-}\cap F_{2}\cap D_{b_{2}}(1/2))\setminus T_{2}$.
\end{itemize}
The high density regions are:
\begin{itemize}
\item $H_{1}=R_{1}\setminus L_{1}$
\item $H_{2}=R_{2}\setminus L_{2}$
\item $H_{3}=A_{2}^{-}\setminus (B_{1}\cup B_{2})$
\item $H_{4}=M^{+}\setminus L_{3}$.
\item $H_{5}=S_{2}$.
\end{itemize}
\begin{figure}[h]
\centering
\includegraphics[height=100mm]{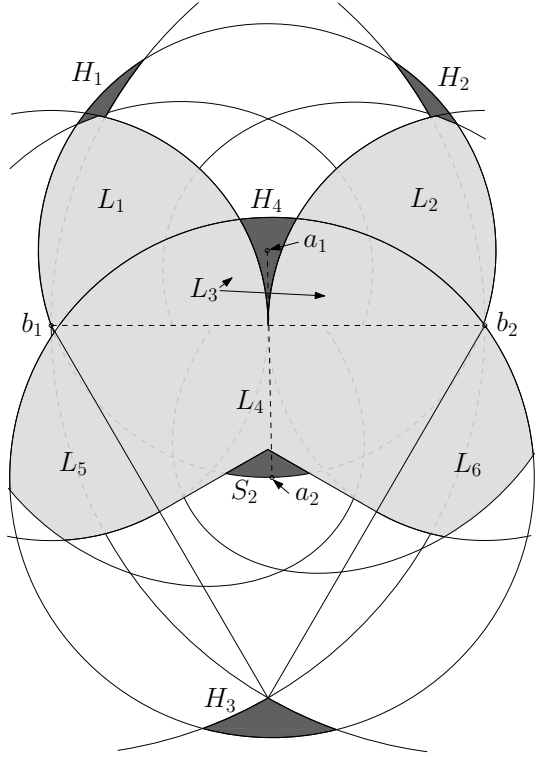}
\caption{The dark shaded region is $H$ and the light shaded region is $L$.}
\label{FigHandL}
\end{figure}
And we write:
\begin{align}
H & = \bigcup_{i=1}^{5}H_{i}\\
L & = \bigcup_{i=1}^{6}L_{i}
\end{align}

See Figure~\ref{FigHandL} for an illustration of this.

We want to show that $L$ is empty, and that $H$ contains at least $k$ points. To do this we will first show that $H\cup L$ contains at least $k$ points and then show that $\#L=0$.

\begin{lem}\label{Hdense}
With the regions as defined above, we have $\#(H\cup L)>k$.
\end{lem}

\begin{proof}
Note that $L_{4}\cup L_{5}\cup L_{6}\supset M^{-}\setminus S_{2}$, and thus: \begin{align}H\cup L\supset R_{1}\cup R_{2}\cup H_{3}\cup M\label{LHDense1}\end{align}
For ease of notation, let $\#(D^{k}(a_{1})\setminus (R_{1}\cup R_{2}\cup M))=\alpha$, $\#(D^{k}(a_{2})\cap B_{1}\cap B_{2})\setminus M=\beta$ and $\# (D^{k}(a_{2})\cap(B_{i}\setminus B_{j}))=\gamma_{i}$, as shown in Figure~\ref{FigRegions1}.

\begin{figure}[h]
\centering
\includegraphics[height=80mm]{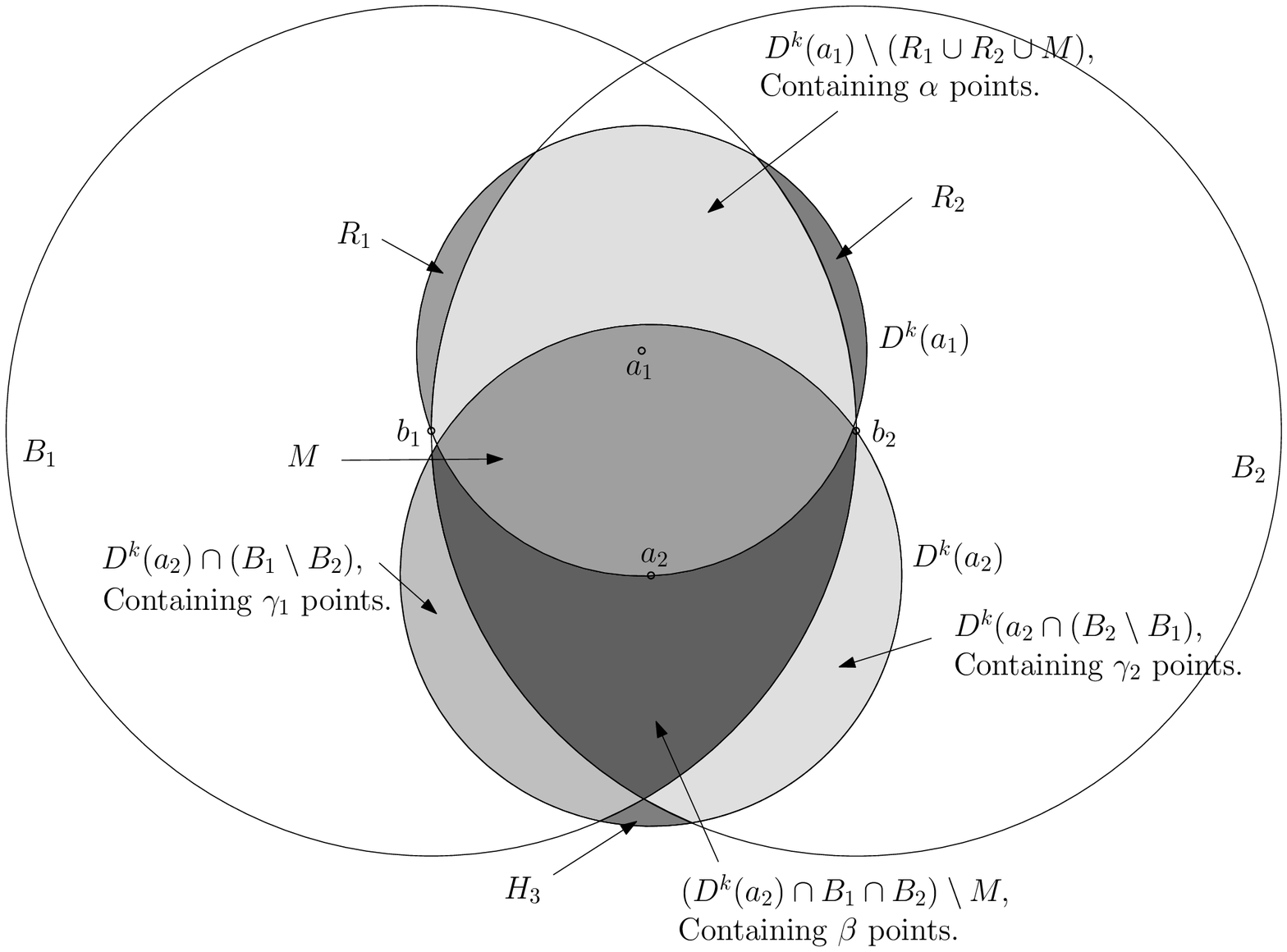}
\caption{The regions we are considering, with their number of points.}
\label{FigRegions1}
\end{figure}

We have the following by counting points in each of the $D^{k}(a_{i})$ (which must contain $k+1$ points) and each of the $B_{i}$ (which can contain at most $k$ points).
\begin{align}
\# R_{1}+ \# R_{2}+ \# M+ \alpha & = k+1\label{A1}\\
\# H_{3}+ \# M + \beta + \gamma_{1}+ \gamma_{2} & = k+1\label{A2}\\
\# R_{1}+ \# M + \alpha+ \beta+ \gamma_{1} & \leq k\label{B1}\\
\# R_{2}+ \# M + \alpha+ \beta+ \gamma_{2} & \leq k\label{B2}
\end{align}

(\ref{A1}) and (\ref{B2}) together tell us that:
\begin{align}
\# R_{1} + \# R_{2} + \# M + \alpha & \geq \# R_{2} + \# M + \alpha + \beta + \gamma_{2}+1\notag
\end{align}
Cancelling terms we get:
\begin{align}
\# R_{1} & \geq \beta + \gamma_{2}+1\label{R1ineq}
\end{align}
Similarly, (\ref{A1}) and (\ref{B1}) imply:
\begin{align}
\# R_{2}\geq\beta+\gamma_{1}+1\label{R2ineq}
\end{align}
Thus, by using (\ref{LHDense1}), (\ref{R1ineq}), (\ref{R2ineq}) and finally (\ref{A2}) we get:
\begin{align}
\#(H\cup L) & \geq\# H_{3} + \# M + \# R_{1} + \# R_{2}\notag\\
& \geq \# H_{3} + \# M + (\beta + \gamma_{2}+1)+(\beta+\gamma_{1}+1)\notag\\
& = (\# H_{3}+ \# M + \beta + \gamma_{1}+ \gamma_{2}) + (\beta + 2)\notag\\
& = k + \beta +3\notag\\
& > k\notag
\end{align}
\end{proof}

We next show that for each $i$, $\#L_{i}=0$.

\begin{lem}
$\#L_{1}=\#L_{2}=\#L_{5}=\#L_{6}=0$.
\end{lem}

\begin{proof}
Lemma~\ref{EllipseLemma} tells us that any point in $L_{1}$ has an out edge to both $a_{1}$ and $b_{1}$, but $L_{1}$ is contained inside both $D^{k}(a_{1})$ and $D^{k}(b_{1})$, and thus must be empty. Similarly for $L_{2}$, $L_{5}$ and $L_{6}$.
\end{proof}

The cases for $L_{3}$ and $L_{4}$ require slightly more work and are dealt with separately.

\begin{lem}\label{L4Empty}
$\# L_{4}=0$
\end{lem}
\begin{proof}
Note that $L_{4}$ is contained in the polygon, $P$, with corners (moving around its perimeter clockwise) at $b_{1}$, $b_{2}$, $u^{-}=(\frac{3}{4},-\frac{\sqrt{3}}{4})$, $w^{-}=(\frac{1}{2},-\frac{1}{2\sqrt{3}})$ and $v^{-}=(\frac{1}{4},-\frac{\sqrt{3}}{4})$. We will show that the left half of this region (namely the convex polygon $P^{l}$, with corners $b_{1}$, $(\frac{1}{2},0)$, $w^{-}$ and $u^{-}$) is contained within $F_{1}$, and then use Lemma~\ref{EllipseLemma} to show that we can have no points in $L_{4}\cap P^{l}$. To do this it is convenient to first bound $S_{2}$ into a convex polygon:

By Lemma~\ref{farapart1}, $a_{1}^{(y)}\geq 0.102$, and thus the minimal possible $y$ co-ordinate of a point $q\in M^{-}$ (and so for $a_{2}$) can be no less than the minimum when taking $a_{1}$ to be at $(1/2,0.102)$ and $D^{k}(a_{1})=A_{1}$. This bounds $q^{(y)}$ (and in particular $a_{2}^{(y)}$) below by:
\[q^{(y)} \geq 0.102 - \sqrt{(1/2)^2+0.102^2}>v^{-(y)}=-\frac{\sqrt{3}}{4}\label{a2ymin}\]
Thus $S_{2}$ is contained in the triangle $T_{a_{2}}$, with corners $u^{-}$, $v^{-}$ and $w^{-}$.

By convexity, to check that $P^{l}\subset F_{1}$ it is enough to check that for every corner of $P^{l}$ and every corner of $T_{a_{2}}$ (labelling these corners by $p_{i}$ and $t_{j}$ respectively) the equation\[\Vert b_{1}p_{i}\Vert +\Vert p_{i}t_{j}\Vert\leq1\]holds. This is the case (calculations omitted), and so $P^{l}\subset F_{1}$.

Lemma~\ref{EllipseLemma} then tells us that any point in $L_{4}\cap P^{l}$ must have an out-edge to both $b_{1}$ and $a_{2}$, but $P^{l}\subset B_{1}$ and $L_{4}\subset D^{k}(a_{2})$, so any point in $L_{4}\cap P^{l}$ would then be joined to both $b_{1}$ and $a_{2}$ in $G$, and so no such point can exist. Similarly, defining $P^{r}$ to be the right half of $P$, $L_{4}\cap P^{r}$ must be empty, and so $\#L_{4}=0$.
\end{proof}

\begin{lem}\label{L3Empty}
The region $L_{3}\cap \{p:p^{(x)}<\frac{1}{2}\}\subset E_{1}$ and $L_{3}\cap \{p:p^{(y)}\geq\frac{1}{2}\}\subset E_{2}$, and so in particular $\# L_{3}=0$.
\end{lem}
\begin{proof}
We show that $L_{3}$ is contained in the polygon $Q$ with corners (moving around its perimeter clockwise) at $b_{1}$, $u^{+}=(\frac{1}{6},\frac{1}{2\sqrt{3}})$,  $v^{+}=(\frac{5}{6},\frac{1}{2\sqrt{3}})$ and $b_{2}$. The proof will then follows as in Lemma \ref{L4Empty}; we show that the left and right halves of $Q$ are contained in $E_{1}$ and $E_{2}$ respectively, and use this to rule out any points in $L_{3}$.

Writing $z^{+}$ for the location $(\tfrac{1}{2},\tfrac{\sqrt{3}}{2})$, we have that $b_{1}\widehat{b_{2}}z^{+}=b_{2}\widehat{b_{1}}z^{+}=\frac{\pi}{3}$. Now, $L_{3}\subset A_{2}^{+}$ (by Lemma~\ref{Properties} part \ref{AsAndBs}), and $a_{2}\widehat{b_{i}}z^{+}\geq\frac{\pi}{2}$ (by Lemma~\ref{Properties} part \ref{Positiona2}), and thus, since $a_{2}\widehat{b_{i}}b_{j}\geq\frac{\pi}{6}$ ($i\neq j$), it follows that $L_{3}$ is contained in the triangle with vertices $b_{1}$, $b_{2}$ and $z^{+}=(\frac{1}{2},\frac{\sqrt{3}}{2})$ (as $L_{3}\subset A_{2}^{+}$). Now, $u^{+}$ and $v^{+}$ lie on the lines $b_{1}z^{+}$ and $b_{2}z^{+}$ respectively, and so we just need to show that $L_{3}$ can't come too high up inside this triangle: By Lemma~\ref{Properties} part \ref{Positiona2}, $a_{2}^{(y)}\leq -\frac{1}{2\sqrt{3}}$, and thus the maximal possible $y$ co-ordinate of a point $q\in M^{+}$ can be no more than the maximum when taking $a_{2}$ to be at $(1/2,-\frac{1}{2\sqrt{3}})$ and $D^{k}(a_{2})=A_{2}$. This bounds $q^{(y)}$ above by: \[q^{(y)} \leq \frac{1}{2\sqrt{3}}\]Thus every point in $M^{+}$, and hence every point in $L_{3}$, is inside $Q$.

By writing $Q^{l}$ for the left half of $Q$, $q_{i}$ for the corners of $Q^{l}$ and noting that $S_{1}$ (and hence $a_{1}$) is contained in the convex polygon $T_{a_{1}}$ with corners $t_{j}$ at $(\frac{1}{2},0)$, $(\frac{\sqrt{3}}{4},\frac{1}{4})$, $w$ and $(1-\frac{\sqrt{3}}{4},\frac{1}{4})$, it follows by convexity that since all of the equations $\Vert b_{1}q_{i}\Vert+\Vert q_{i}t_{j}\Vert\leq 1$ hold, $Q^{l}\subset E_{1}$. Lemma~\ref{EllipseLemma} and the definition of $L_{3}$ then tell us we can have no points inside $L_{3}\cap Q^{l}$. Similarly we can have no points in $L_{3}\cap Q^{r}$, where $Q^{r}$ is the right half of $Q$, and so $\# L_{3}=0$.
\end{proof}

Putting Lemmas~\ref{Hdense}--\ref{L3Empty} together we have:

\begin{lem}\label{HandLLemma}
$\#H\geq k$ and $\#L=0$.\flushright{$\square$}
\end{lem}

\subsubsection{Bounding the relative areas of $H$ and $L$ and the proof of Theorem~\ref{nocrossing}}

We define $\rho_{1}$ and $\rho_{2}$ to be the radius of $D^{k}(a_{1})$ and $D^{k}(a_{2})$ respectively and now move on to bound the relative areas of $H$ and $H\cup L$. However, the regions defined above are quite complicated in shape, and so computing the relative areas, even for particular positions of $a_{1}$ and $a_{2}$ and given values of $\rho_{1}$ and $\rho_{2}$, involves some complicated integrals. Moreover, we need to bound the relative areas over all possible positions of $a_{1}$ and $a_{2}$ and all allowable values of $\rho_{1}$ and $\rho_{2}$. To obtain a bound we will thus break things down into finite cases as follows:

We first tile $S_{n}$ with small squares and then consider the possible pairs of tiles which can contain $a_{1}$ and $a_{2}$. For each such pair, we will bound $|H|$ above and $|L|$ below, and thus bound $H$ above and $|L|$ below absolutely over all positions of $a_{1}$ and $a_{2}$.

Practically, this requires the use of a computer, but will still be completely rigorous.

To make the calculations as simple as possible, we wish to reduce the number of variables we have to maximise and minimise over. In light of this we split $L$ and $H$ into two parts, each of whose size will be dependent on the position of only one of $a_{1}$ and $a_{2}$ (we will show this on a case by case basis later); namely $L$ splits into $L^{+}=L_{1}\cup L_{2}\cup L_{3}$ and $L^{-}=L_{4}\cup L_{5}\cup L_{6}$ and $H$ splits into $H_{1}\cup H_{2}\cup S_{2}$ and $H_{3}\cup H_{4}$. Further, it is easy to see that for any fixed positions of $a_{1}$ and $a_{2}$, the area of any part of $H$ will be maximised by maximising $\rho_{1}$ and $\rho_{2}$, and that the area of any part of $L$ will be minimised by minimising $\rho_{1}$ and $\rho_{2}$. Thus, for each of the given parts of $H$ or $L$ above, we need only to bound the integral over the position of one of $a_{1}$ and $a_{2}$ and nothing else.

Our exact method is as follows: We tile $S_{n}$ with small squares of side length $s$, which are aligned with the edge $b_{1}b_{2}$, i.e. $b_{1}b_{2}$ will run along the edges of all the square it touches, and both $b_{1}$ and $b_{2}$ will be on the corners of squares (to prove our bound, we will use a square side length of $s=0.001\Vert b_{1}b_{2}\Vert$). Whilst bounding an area dependent on the position of $a_{i}$, and given some small square $X$ with centre $x$, we define $\sigma^{X}_{i}$ and $\rho^{X}_{i}$ to be the minimum and maximum values of $\rho_{i}$ over all possible positions of $a_{i}$ within $X$. We can then bound the area of the relevant part of $H$ above by simply counting every square that could be within the part of $H$ that contains any location within $\rho^{X}_{i}$ of any location in $X$, and bound the area of the relevant part of $L$ below by counting only squares that are entirely within that part of $L$ and are entirely within $\sigma^{X}_{i}$ of every location within $X$. In fact, it suffices to count every square that has its centre within $\rho^{X}_{i}+s\sqrt{2}$ of $x$ for the bound on $H$, and only squares that have their centres within $\sigma^{X}_{i}-s\sqrt{2}$ of $x$ for the bound on $L$, since this can only weaken the bounds obtained. We can then bound the areas of the relevant parts of $H$ and $L$ above and below respectively by taking the maximum and minimum of these sums over every square that could possibly contain $a_{i}$.

Since the regions we are using are often dependent on the ellipses $E_{i}$ and $F_{i}$, and these are dependent on the position of $a_{1}$ and $a_{2}$, it is useful to define:\[E_{i}^{X}=\{q\in S_{n}:\underset{a\in X}{\text{max}}\,\Vert b_{i}q\Vert+\Vert aq\Vert\leq1\}\]Similarly we define $F_{i}^{X}$ when $a_{2}\in X$. Thus $E_{i}^{X}$ is the intersection of the $E_{1}(a_{1})$ over all possible positions of $a_{1}$ within $X$. It is worth noting that when a region in $L$ depends on an ellipse, it is contained within the ellipse, and when a region in $H$ depends on an ellipse, it is outside the ellipse, so we will always want to use the intersection of the possible ellipses to bound our area, rather than a union. Note also that any small square $Y$, with centre $y$, such that $\Vert b_{i}y\Vert+\Vert xy\Vert\leq 1-\frac{3\sqrt{2}}{2}s$, will be entirely contained within $E_{i}^{X}$.

\begin{lem}\label{L+Lemma}
$|L^{+}|>0.3411$
\end{lem}
\begin{proof}
Note that:
\begin{align}
L^{+} & =L_{1}\cup L_{2}\cup L_{3}\label{L+Eq1}\\
& = \left(D^{k}(a_{1})^{+}\cap E_{1}\cap D_{b_{1}}(\tfrac{1}{2})\right)\cup\left(D^{k}(a_{1})^{+}\cap E_{2}\cap D_{b_{2}}(\tfrac{1}{2})\right)\label{L+Eq2}\\
& = D^{k}(a_{1})^{+}\cap\left[\left(E_{1}\cap D_{b_{1}}(\tfrac{1}{2})\right)\cup\left(E_{2}\cap D_{b_{2}}(\tfrac{1}{2})\right)\right]
\end{align}
Where (\ref{L+Eq2}) follows from (\ref{L+Eq1}) by Lemma~\ref{L3Empty}. Thus $|L^{+}|$ does not depend on $a_{2}$, and so is a function of the position of $a_{1}$ and $\rho_{1}$ only.

We know that $D^{k}(a_{1})$ must contain $a_{2}$ as well as at least one point in $H_{1}$ (i.e. in $R_{1}$ and outside of $E_{1}\cap D_{b_{1}}(1/2)$) and at least one point in $H_{2}$ (i.e. in $R_{2}$ and outside of $E_{2}\cap D_{b_{2}}(1/2)$). Call the closest locations to $a_{1}$ in $H_{1}$ and $H_{2}$, $h_{1}$ and $h_{2}$ respectively, and note that they are dependent only on the position of $a_{1}$.

Now, given that $a_{1}$ is in some small square $X$ with centre $x$, we set $h_{1}^{X}$ to be the lower down (on $\partial B_{2}=\partial D_{b_{2}}(1)$) of the two location $\partial B_{2}\cap \partial D_{b_{1}}(1/2)$ and the location $q$ on $\partial B_{2}$ for which $\Vert b_{1}q\Vert+\Vert xq\Vert=1-\frac{\sqrt{2}}{2}s$, and similarly define $h_{2}^{X}$. Thus $h_{1}^{X}$ (correspondingly $h_{2}^{X}$) is at least as far down $\partial B_{2}$ (correspondingly $\partial B_{1}$) as $h_{1}$ (or $h_{2}$) for any position of $a_{1}$ within $X$. Thus we define: \[\rho=\text{max}\{\Vert xh_{1}^{X}\Vert,\Vert xh_{2}^{X}\Vert,\Vert xa_{2}\Vert\}-\frac{\sqrt{2}}{2}s\leq\sigma_{1}^{X}\]
Then a small square $Y$ with centre $y$ will be entirely within $L^{+}$ regardless of where in $X$ $a_{1}$ lies, so long as:
\begin{itemize}
\item $Y$ is entirely above the line $b_{1}b_{2}$,
\item $\Vert yx\Vert\leq\rho-s\sqrt{2}$ (note that $s\tfrac{\sqrt{2}}{2}$ is subtracted twice from $\rho_{\text{min}}^{X}$ to account for the possible locations of points within both of the squares $X$ and $Y$) and finally,
\item every point in $Y$ is inside both $D_{b_{1}}(1/2)$ and $E_{1}^{X}$ or every point in $Y$ is inside both $D_{b_{2}}(1/2)$ and $E_{2}^{X}$.
\end{itemize}
See Figure~\ref{FigLPlus}.
\begin{figure}[ht]
\centering
\includegraphics[height=70mm]{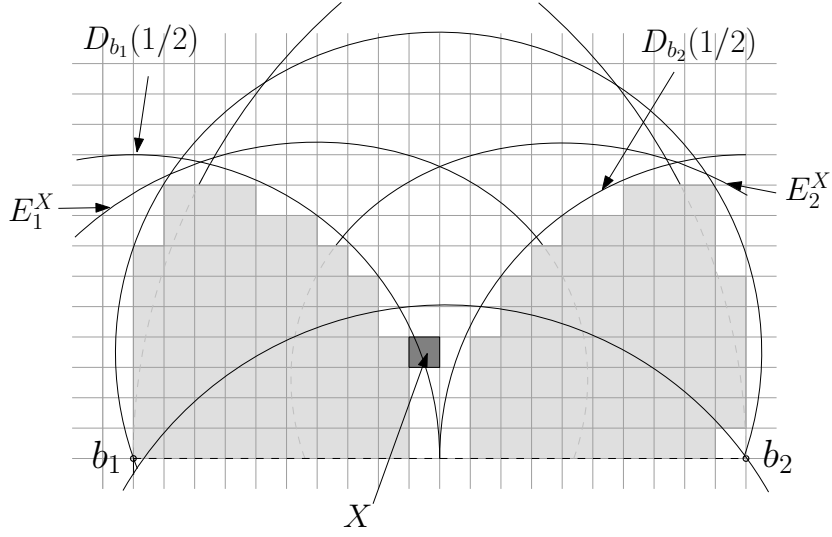}
\caption{An incidence of the squares that will be counted as being in $L^{+}$.}
\label{FigLPlus}
\end{figure}

Performing our numerical integration on a computer then gives us $|L^{+}|>0.3411\ldots$ with the minimum achieved when $a_{1}$ was in either of the squares with centres at $(0.4995,0.1895)$ and $(0.5005,0.1895)$.
\end{proof}

\begin{lem}\label{L-Lemma}
$|L^{-}|>0.3564$
\end{lem}
\begin{proof}
Note that:
\begin{align}
L^{-} & = L_{4}\cup L_{5}\cup L_{6}\notag
\end{align}
None of the definitions of $L_{4}$, $L_{5}$ or $L_{6}$ are dependent of the position of $a_{1}$ or the value of $\rho_{1}$, although the region where we can place $a_{2}$ (i.e. the region $S_{2}$) is dependent on $a_{1}$. From Lemma~\ref{farapart1} we know that we cannot have $a_{1}$ as low as the point $(\frac{1}{2},\frac{1}{4\sqrt{6}})$, and so, using Lemma~\ref{Properties} we may assume $a_{1}$ is at $(\frac{1}{2},\frac{1}{4\sqrt{6}})$ and $\rho_{1}$ is maximal when determining if a small square contains a possible location in $S_{2}$.

Given that $a_{2}$ is in some small square $X$ with centre $x$, we can define:\[\sigma=\text{max}\{\Vert xa_{1}\Vert,\Vert xz\Vert\}-\frac{\sqrt{2}}{2}s\leq\sigma_{2}^{X}\] 
Then a small square $Y$, with centre $y$, will be entirely within $L^{-}$ regardless of where in $X$ $a_{2}$ lies, so long as:
\begin{itemize}
\item $Y$ is entirely below the line $b_{1}b_{2}$,
\item $\Vert yx\Vert\leq\sigma-s\sqrt{2}$,
\item every point $q\in Y$:\begin{enumerate}
	\item is inside both $D_{b_{1}}(1/2)$ and $F_{1}^{X}$,
	\item or is inside $D_{b_{2}}(1/2)$ and $F_{2}^{X}$,
	\item or has $q\in T_{2}$ and either $b_{1}\widehat{b_{2}}q<\frac{\pi}{6}$ or $b_{2}\widehat{b_{1}}q<\frac{\pi}{6}$.
                           \end{enumerate}
\end{itemize}
Computer calculations then gives $|L^{-}|>0.3564\ldots$ with a minimum value achieved when $a_{2}$ was in either of the squares with centres at $(0.4995,-0.3825)$ and $(0.5005,-0.3825)$.
\end{proof}

\begin{lem}\label{H+Lemma}
$|H_{1}\cup H_{2}\cup S_{2}|<0.1300$.
\end{lem}
\begin{proof}
The areas of $H_{1}$, $H_{2}$ and $S_{2}$ all depend only on the position of $a_{1}$ and the value of $\rho_{1}$, and thus to bound their union above we may assume that $a_{2}$ is located at $(\frac{1}{2},-\frac{1}{2\sqrt{3}})$ and $\rho_{2}$ is maximal, as in lemma \ref{L+Lemma}. We know also that $D^{k}(a_{1})\subset A_{1}$, so that neither $b_{1}$ nor $b_{2}$ are within $\rho_{1}$ of $a_{1}$.

Given that $a_{1}$ is in some small square $X$ with centre $x$, the above tells us that, defining:\[\tau=\text{min}\{\Vert b_{1}x\Vert,\Vert b_{2}x\Vert\}+\frac{\sqrt{2}}{2}s\geq\rho_{1}^{X}\]
Then a small square $Y$, with centre $y$, can have some part of itself in $H_{1}$, $H_{2}$ or $S_{2}$ only if:
\begin{itemize}
\item $\Vert yx\Vert\leq\tau+s\sqrt{2}$ and
\item we have one of the following:	\begin{enumerate}
	\item Any location in $Y$ is inside $R_{1}$ and outside of either $E_{1}^{X}$ or $D_{b_{1}}(1/2)$ ($Y$ contains a location in $H_{1}$)
	\item Any location in $Y$ is inside $R_{2}$ and outside of either $E_{2}^{X}$ or $D_{b_{2}}(1/2)$ ($Y$ contains a location in $H_{2}$)
	\item Any location $q\in Y$ has $b_{1}\widehat{b_{2}}q\geq\frac{\pi}{6}$ and $b_{2}\widehat{b_{1}}q\geq\frac{\pi}{6}$ ($Y$ contains a location in $S_{2}$).
                                 	\end{enumerate}
\end{itemize}
Computer calculations then give $|H_{1}\cup H_{2}\cup S_{2}|<0.1299\ldots$ with a maximum achieved when $a_{1}$ was in the square with centre at $(0.4995,0.2885)$.
\end{proof}

\begin{lem}\label{H-Lemma}
$|H_{3}\cup H_{4}|<0.0958$.
\end{lem}
\begin{proof}
The areas of $H_{3}$ and $H_{4}$ depend only on the position of $a_{2}$ and the value of $\rho_{2}$, and that when calculating whether a small square could contain a location in $S_{2}$, we may assume that $a_{1}$ is at $(\frac{1}{2},\frac{1}{4\sqrt{6}})$ and $\rho_{1}$ is maximal, as in Lemma~\ref{L-Lemma}.

Given that $a_{2}$ is in some small square $X$ with centre $x$, the above tells us that, defining:\[\upsilon=\text{min}\{\Vert b_{1}x\Vert,\Vert b_{2}x\Vert\}+\frac{\sqrt{2}}{2}s\geq\rho_{2}^{X}\]
Then a small square $Y$ with centre $y$ can have some part of itself in $H_{3}$ or $H_{4}$ only if:
\begin{itemize}
\item $\Vert yx\Vert\leq\upsilon+s\sqrt{2}$ and
\item either of the following holds:\begin{enumerate}
	\item Any location in $Y$ is outside $B_{1}\cup B_{2}$ ($Y$ contains a location in $H_{3}$)
	\item Any location in $Y$ is above the line $b_{1}b_{2}$ and is outside $D_{b_{1}}(1/2)\cup D_{b_{2}}(1/2)$ ($Y$ contains a location in $H_{4}$)
                                 \end{enumerate}
\end{itemize}
Our computer calculations gives us that $|H_{4}\cup H_{4}|<0.0957\ldots$ with a maximum achieved when $a_{2}$ was in the square with centre at $(0.4995,-0.4335)$.
\end{proof}

We can use Lemmas~\ref{L+Lemma}-\ref{H-Lemma} to bound the ratio $\frac{|H|}{|H\cup L|}$:

\begin{lem}\label{HandLSmall}
$\frac{|H|}{|H\cup L|}<0.2446$.
\end{lem}
\begin{proof}
Note that since $H$ and $L$ are disjoint, $\frac{|H|}{|H\cup L|}=\frac{|H|}{|H|+|L|}$, which is strictly increasing in $|H|$ and decreasing in $|L|$. Thus, by using Lemmas~\ref{L+Lemma}-\ref{H-Lemma} we have:
\begin{align}
\frac{|H|}{|H\cup L|} & < \frac{0.1300+0.0958}{0.1300+0.0958+0.3411+0.3564}\notag\\
& < 0.2446\notag
\end{align}
\end{proof}

Using all of the above, we can finally prove Theorem~\ref{nocrossing}:

\vspace{0.5cm}

\begin{proofof}{Theorem~\ref{nocrossing}}
We pick six points $a_{1}$, $a_{2}$, $b_{1}$, $b_{2}$, $a_{1}^{(k)}$ and $a_{2}^{(k)}$, and write $Z$ for the event that $a_{1}$, $a_{2}$, $b_{1}$ and $b_{2}$ form a crossing pair, and that $a_{1}^{(k)}$ and $a_{2}^{(k)}$ are the $k^{th}$ nearest neighbours of $a_{1}$ and $a_{2}$ respectively.

When $Z$ occurs, these six points define the regions $H$ and $L$, and so for any given six tuple of points, Lemmas~\ref{HandLLemma} and \ref{HandLSmall} tell us:

\begin{align}
\Prb (Z) & \leq\left(\frac{|H|}{|H\cup L|}\right)^{k}\notag\\
& < n^{c\log 0.2446}
\end{align}

Now, there are $O(n)$ choices for $a_{1}$, and once this has been chosen there are only $O(\log n)$ choices for each of $a_{2}$, $b_{1}$, $b_{2}$, $a_{1}^{(k)}$ and $a_{2}^{(k)}$ (since all five have either an out edge to or from $a_{1}$ (except for $a_{2}^{k}$ which must have an out edge from $a_{2}$), and so must be within $O(\sqrt{\log n})$ of $a_{1}$ by Lemma~\ref{edgelengths}). Thus there are $O(n\log^{5} n)$ choices for our system, and so, with high probability, no two edges in different components cross so long as:
\begin{align}
c\log 0.2446 & < -1\notag
\end{align}
or equivalently:
\begin{align}
c > 0.7102\notag
\end{align}
\end{proofof}

\subsection{There can only be one large component}

We use Lemma~\ref{farapart2} and Theorem~\ref{nocrossing} to get a bound on the absolute distance between any two edges in different components:

\begin{cor}
If $k=c\log n$, and $c>0.7102$, then with high probability the minimal distance between two edges in different components is at least $r/5$, where $r$ is as given in Lemma~\ref{edgelengths}.
\end{cor}
\begin{proof}
Since $c>0.7102$ we may assume, by Theorem~\ref{nocrossing}, that no two edges in different components cross. Thus the minimal distance between two such edges will be at the end point of one of them. Corollary~\ref{farapart2} then gives us the result.
\end{proof}

Using the above, we now meet all of the conditions for Lemma 12 of [\ref{MW}] so long as $k>0.7102\log n$, except that now the minimal distance between edges in different components is $r/5$ instead of $r/2$, however this requires only trivial changes in the proof, and so we gain:

\begin{prop}\label{PropOneBigComponent}
For fixed $c>0.7102$, if $k>c\log n$, then there exists a constant $c'$ such that the probability that $G_{n,\lfloor c \log n\rfloor}$ contains two components of (Euclidean) diameter at least $c'\sqrt{\log n}$ tends to zero as $n\rightarrow\infty$.\flushright{$\square$}
\end{prop}

\section{The main result}

\subsection{Approach and simple bound}
Using the results from the previous section we can now proceed to gain an upper bound for the threshhold for connectivity by ruling out the chance of having a small component.

We wish to prove a good bound on the critical constant $c$ such that if $k>c\log n$ then $\Prb(G_{n,k}\textrm{ disconnected})\rightarrow0$ as $n\rightarrow\infty$. Proposition~\ref{PropOneBigComponent} tells us that if $G$ is not connected, and $k>0.7102\log n$, then we may assume that there is a small component somewhere. In the next section we will show that such a small component will not exist with high probability for $c>0.9684$, but first illustrate a simpler proof that works for $c>1.0293$ to give the general approach. This proof is similar to the first part of Theorem 15 of [\ref{MW}]. We start by introducing some notation:

\begin{definition}
Let $d$ be $\max\{c',4\sqrt{c_{+}/\pi},\frac{1}{4\sqrt{c_{-}/\pi}},1\}$, (where $c_{+}$ and $c_{-}$ are the constants from Lemma~\ref{edgelengths}, and $c'$ is the constant given by Proposition~\ref{PropOneBigComponent}).

Given four points, $a$, $b$, $x_{l}$ and $x_{r}$ in $S_{n}$, we define $\rho=\Vert ab\Vert$ and, writing $D^{l}_{x}(y)$ and $D^{r}_{x}(y)$ for the left and right half-disks of radius $y$ centred on $x$, we define the regions:
\begin{itemize}
\item $C=\left(D^{l}_{x_{l}}(\rho)\cup D^{r}_{x_{r}}(\rho)\right)\cap S_{n}$,
\item $A=\left(D_{a}(\rho)\setminus \left(D_{b}(\rho)\cup C\right)\right)\cap S_{n}$, and
\item $B=\left(D_{b}(\rho)\setminus \left(D_{a}(\rho)\cup C\right)\right)\cap S_{n}$.
\end{itemize}
See Figure~\ref{FigBasicBound} for an illustration of these regions.

We say that $a$, $b$, $x_{l}$ and $x_{r}$ form a \emph{component set-up} if:
\begin{enumerate}
\item The points $b$, $x_{l}$ and $x_{r}$ are all within $d\sqrt{\log n}$ of $a$,\label{CSClose}
\item $\#C=0$,\label{CSEmpty}
\item and at least one of $\#A\geq k$ and $\#B\geq k$ holds.\label{CSFull}
\end{enumerate}
\begin{figure}[ht]
\centering
\includegraphics[height=70mm]{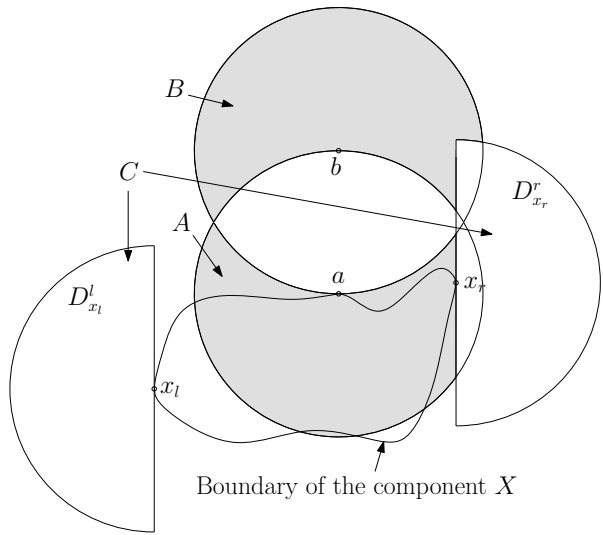}
\caption{The set up of the points $a$, $b$, $x_{l}$ and $x_{r}$ and the regions they define.}
\label{FigBasicBound}
\end{figure}
\end{definition}

\begin{lem}\label{BBRegions}
If there is a component, $X$, of diameter at most $d\sqrt{\log n}$ in $G$, then with high probability some four points form a component set-up.
\end{lem}
\begin{proof}
Let $a\in X$ and $b\notin X$ be such that they minimise $\Vert ab\Vert$ over all such pairs. Let $x_{l}$ be the left most point in the component $X$ and $x_{r}$ the right most point. We show that these four points form a component set-up with high probability.

Since $\diam(X)\leq d\sqrt{\log n}$, $x_{l}$ and $x_{r}$ are within $d\sqrt{\log n}$ of $a$, and Lemma~\ref{edgelengths} tell us that $b$ is within $d\sqrt{\log n}$ of $a$ with high probability, so Condition~\ref{CSClose} holds with high probability. For any $z\in X$ we cannot have any points in $D_{z}(\rho)$ that are not in $X$, by the minimality of $\Vert ab\Vert$, and so in particular $C$ is empty, i.e. Condition~2 is met. Finally, since $ab\notin G$ and since $D_{a}(\rho)\cap D_{b}(\rho)$ is empty by the minimality of $\Vert ab\Vert$, there must be at least $k$ points in at least one of $A$ or $B$, so Condition~3 is met.
\end{proof}

We will show that if $k=c\log n$ and $c>1.0293$, then with high probability no quadruple forms a component set-up, at which point Lemma~\ref{BBRegions} tells us there will be no small component in $G$ with high probability.

\begin{lem}\label{EasyBound'}If:
\begin{align}
 c&>\log\left(\frac{8\pi+3\sqrt{3}}{2\pi+3\sqrt{3}}\right)^{-1}\approx 1.0293\notag
\end{align}
and $k=c\log n$, then, with high probability, no quadruple $(a,b,x_{l},x_{r})$ with all of $a$, $b$, $x_{l}$ and $x_{r}$ at least $d\sqrt{\log n}$ from the boundary of $S_{n}$ form a component set-up.
\end{lem}
\begin{proof}
We will show that if we pick four points in $S_{n}$; $a$, $b$, $x_{l}$ and $x_{r}$ that are all within $d\sqrt{\log n}$ of $a$ (i.e. meet Condition~\ref{CSClose} of being a component set-up), then the probability, $p(n)$, that they meet Conditions~\ref{CSEmpty} and \ref{CSFull} of being a component set-up decays as at least $n^{-(1+\varepsilon)}$ for some $\varepsilon>0$. Then, since there are only $\Or(n)$ points in $S_{n}$ in total (with high probability), and since all four points are within $d\sqrt{\log n}$ of $a$, Lemma~\ref{edgelengths} tells us that there are only $\Or(n(\log n)^{3})$ choices for such a system, and so, with high probability, no four points form a component set-up.

Since $x_{l}$ and $x_{r}$ are at least $d\sqrt{\log n}$ from the boundary of $S_{n}$, and $\rho=\Vert ab\Vert\leq d\sqrt{\log n}$, we have that $|C|=\pi\rho^{2}$. We also know that  $|A|,|B|\leq(\pi/3+\sqrt{3}/2)\rho^{2}$, and so, by Lemma~\ref{Full-Empty}:
\begin{align}
p(n) & \leq \Prb(\#C=0\textrm{ and }\#A\geq k)+\Prb(\#C=0\textrm{ and }\#B\geq k)\notag\\
& \leq \left(\frac{|A|}{|A\cup C|}\right)^{k} + \left(\frac{|B|}{|B\cup C|}\right)^{k}\notag\\
& \leq 2\left(\frac{(\pi/3+\sqrt{3}/2)\rho^{2}}{\pi\rho^{2}+(\pi/3+\sqrt{3}/2)\rho^{2}}\right)^{k}\notag\\
& = 2\left(\frac{2\pi+3\sqrt{3}}{8\pi+3\sqrt{3}}\right)^{k}\notag\\
& = 2\Exp \left(-c\log\left(\frac{8\pi+3\sqrt{3}}{2\pi+3\sqrt{3}}\right)\log n\right)\label{BBEq1}
\end{align}
If $c>\log\left(\frac{8\pi+3\sqrt{3}}{2\pi+3\sqrt{3}}\right)^{-1}$, then (\ref{BBEq1}) is at most $2n^{-(1+\varepsilon(c))}$ for some $\varepsilon(c)>0$, and so we are done.
\end{proof}

We now rule out having a component set-up near the edge of $S_{n}$, and so having a small component near the edge of $S_{n}$. The bound we prove here will also be strong enough to rule out the edge case in our stronger bound on the connectivity threshhold that we give in the next section.

\begin{lem}\label{NoBoundaries}\mbox{}
\begin{enumerate}
\item\label{NBCor} If $c>0$ and $k=c\log n$, then with high probability there is no component set-up containing a point within $2d\sqrt{\log n}$ of a corner of $S_{n}$.
\item\label{NBEdge} If $c>0.8343$ and $k=c\log n$, then with high probability there is no component set-up containing a point within $d\sqrt{\log n}$ of any edge of $S_{n}$.\end{enumerate}
\end{lem}
\begin{proof}
The proof proceeds almost exactly as in the previous lemma. We again pick our four points $a$, $b$, $x_{l}$ and $x_{r}$ with $b$, $x_{l}$ and $x_{r}$ within $d\sqrt{\log n}$ of $a$ and bound the probability that they meet Conditions~\ref{CSEmpty} and \ref{CSFull} of forming a component-set-up. We write $p_{c}(n)$ and $p_{e}(n)$ for the probabilities of these events for a quadruple near a corner and an edge respectively.
\begin{Parts}
\item The number of such quadruples with at least one point within $2d\sqrt{\log n}$ of a corner is $\Or((\log n)^{4})$. We show that $p_{c}(n)$ decays as at least $n^{-\varepsilon}$, for some $\varepsilon>0$.

We will have that $|A|,|B|\leq(\pi/3+\sqrt{3}/2)\rho^{2}$ (where again $\rho=\Vert ab\Vert$).

If one of our points is within $d\sqrt{\log n}$ of a corner of $S_{n}$ we must still have $|C|\geq\pi/4$, and so, using Lemma~\ref{Full-Empty}:
\begin{align}
p_{c}(n) & \leq \Prb(\#C=0\text{ and }\#A\geq k)+\Prb(\#C=0\text{ and }\#B\geq k)\notag\\
& \leq\left(\frac{|A|}{|A|+|C|}\right)^{k}+\left(\frac{|B|}{|B|+|C|}\right)^{k}\notag\\
& < 2\left(\frac{(\pi/3+\sqrt{3}/2)\rho^{2}}{(\pi/4)\rho^{2}+(\pi/3+\sqrt{3}/2)\rho^{2}}\right)^{c\log n}\notag\\
& < 2n^{-0.3439c}\label{EBEq1}
\end{align}
And thus for any $c>0$ the exponent of (\ref{EBEq1}) is strictly less than zero, and so with high probability there are no small components containing a point within $d\sqrt{\log n}$ of any corner of $S_{n}$.

\item The number of such quadruples with at least one point within $d\sqrt{\log n}$ of an edge is $\Or(\sqrt{n}(\log n)^{3})$. We show that $p_{e}(n)$ decays as at least $n^{-(1/2+\varepsilon)}$, for some $\varepsilon>0$.
If none of our points are within $2d\sqrt{\log n}$ of a corner, but at least one is within $2d\sqrt{\log n}$ of an edge, then $|C|\geq\frac{\pi}{2}\rho^{2}$ (either we have all of one of the half disks $D_{x_{l}}^{l}$ and $D_{x_{r}}^{r}$ or at least half of each), and so:
\begin{align}
p_{e}(n) & \leq \left(\frac{|A|}{|A|+|C|}\right)^{k}+\left(\frac{|B|}{|B|+|C|}\right)^{k}\notag\\
& < 2\left(\frac{(\pi/3+\sqrt{3}/2)\rho^{2}}{(\pi/2)\rho^{2}+(\pi/3+\sqrt{3}/2)\rho^{2}}\right)^{c\log n}\notag\\
& < 2n^{-0.5993c}\label{EBEq2}
\end{align}
For any $c>0.8343$ the exponent of (\ref{EBEq2}) is strictly less than $-\tfrac{1}{2}$ and so we are done.
\end{Parts}
\end{proof}

Putting together Lemmas~\ref{EasyBound'} and \ref{NoBoundaries}, and applying Lemma~\ref{BBRegions} and Proposition~\ref{PropOneBigComponent}, we have:
\begin{prop}\label{EasyBound}
Let $p(n)$ be the probability that $G_{n,k}$ is disconnected, then, provided $k=c\log n$ and:
\begin{align}
 c&>\log\left(\frac{8\pi+3\sqrt{3}}{2\pi+3\sqrt{3}}\right)^{-1}\approx 1.0293\notag
\end{align}
we have:
\[p(n)\rightarrow 0,\textrm{ as }n\rightarrow\infty\]
\end{prop}

\subsection{The Size of Small Components\\ and an Improved Bound}
The previous section gives a reasonably good upper bound on the connectivity threshold for $G_{n,k}$, so that we know if $k>1.0293\log n$, then $G_{n,k}$ is connected with high probability. The best lower bound known is that if $k<0.7209\log n$ then $G_{n,k}$ is disconnected with high probability, which follows from Balister, Bollob\'{a}s, Sarkar and Walter's bound on the directed model [\ref{MW}]. This leaves the question: could the connectivity threshold be exactly $k=\log n$? We show that this hypothesis, which was conjectured originally by Xue and Kumar for the original undirected model [\ref{XandK}], and is true in the Gilbert model, does not hold here, thus further disproving their conjecture, since the threshold for the strict undirected model must be at least as high as that in the original undirected model. In particular we show that if $k>0.9684\log n$ then $G$ is connected with high probability.

To show this improved bound, we first show that the small components in $G$ (i.e. of diameter $\Phi(\log n)$) contain far fewer than $k$ points as $k$ approaches the lower bound on the connectivity threshold, and then use this to improve our upper bound. One major tool that we use in this section is an isoperimetric argument. As in [\ref{MW2}] this will allow us to bound the empty area around any small component as a function of how much space that component takes up. We use the isoperimetric theorem in its following form, which is a consequence of the Brunn-Minkowski inequality, see e.g. [\ref{BMI}]. Part 2 of the Lemma follows from an easy reflection argument.

\begin{lem}\label{IsoLem1}\mbox{}
\begin{enumerate}
\item For any $\lambda>0$ the subset $A$ of the plane of area $\lambda$ that minimises the area of the $\delta$-blowup, $A(\delta)$ (the subset of the plane within $\delta$ of any location in $A$), is the disc of area $\lambda$.
\item The subset $A$ on the half plane $E^{+}$ of area $\lambda$ that minimises the area of the intersection of $A(\delta)$ and $E^{+}$ is the half disc of area $\lambda$ centred along the edge of $E^{+}$.
\end{enumerate}
\end{lem}

To use Lemma~\ref{IsoLem1}, we follow [\ref{MW2}] and tile $S_{n}$ with a fine square grid. We can then look at the number of tiles that a small component hits to give a bound on the empty area around it. To be precise:

We set $M=20000d$ (a large enough value to gain a good result) and tile $S_{n}$ with small squares of side length $s=\sqrt{\log n}/M$. We form a graph $\widehat{G}$ on these tiles by joining two tiles whenever the distance between their centres is at most $2d\sqrt{\log n}$. We call a pointset \emph{bad} if any of the following hold (and \emph{good} otherwise):
\begin{enumerate}
\item there exist two points that are joined in $G$ but the tiles containing these points are not joined in $\widehat{G}$,
\item there exist two points at most distance $\tfrac{1}{d}\sqrt{\log n}$ apart that are not joined,
\item there exists a half-disc based at a point of $G$ of radius $d\sqrt{\log n}$ that is contained entirely within $S_{n}$ and contains no (other) point of $G$,
\item there exists two components in $G_{n,k}$ with Euclidean diameter at least $d\sqrt{\log n}$,
\item there exists a component of diameter at most $d\sqrt{\log n}$ containing a vertex within distance $2d\sqrt{\log n}$ of a corner of $S_{n}$.
\item there exists two different components $X$ and $Y$ such that an edge in component $X$ crosses an edge in component $Y$.
\end{enumerate}
Note that unlike in [\ref{MW2}], we do not insist that a small component cannot be near an edge of $S_{n}$, but only that it can't be near a corner, since our Lemma~\ref{NoBoundaries} is not strong enough to rule out the existence of small components near the edge of $S_{n}$ around the lower bound on the connectivity threshold ($k=0.7209\log n$).

\begin{lem}\label{GoodLem}
If $k=c\log n$ and $c>0.7102$, then with high probability the configuration is good.
\end{lem}

\begin{proof}\mbox{}
\begin{itemize}
\item By our choice of~$d$ and Lemma~\ref{edgelengths} Conditions~1, 2 and 3 hold with high probability.
\item For $k>0.7102\log n$, Proposition~\ref{PropOneBigComponent} ensures Condition 4 holds with high probability.
\item Lemma~\ref{NoBoundaries} part~1 ensures Condition 5 holds with high probability.
\item For $k>0.7102\log n$, Theorem~\ref{nocrossing} ensures Condition 6 holds with high probability.
\end{itemize}
Since each condition holds with high probability, they will all hold together with high probability, and so the configuration will be good with high probability.
\end{proof}

We will consider what can happen around a small component once we know which tiles the component meets. We make the following definitions:

\begin{definition}
Given two points, $a$, $b$, and a collection of tiles $Y$ with $a\in Y$ and $b\notin Y$, we define, as before, $\rho=\Vert ab\Vert$ and $A=\left(D_{a}(\rho)\setminus D_{b}(\rho)\right)\cap S_{n}$, and define the regions:
\begin{itemize}
\item $Z$ to be all tiles not in $Y$ with their centre within $\rho-\sqrt{2}s$ of the centre of a tile in $Y$,
\item $B'$ to be $D_{b}(\rho)\setminus (D_{a}(\rho)\cup Y\cup Z)$, and
\item $Y'$ to be the tiles in $Y$ that have their centre within $\rho+\sqrt{2}s$ of $a$ (so that the tiles in $Y$ that meet the region $A$ defined previously are all in $Y'$).
\end{itemize}
See Figure~\ref{FigBasicTile} for an illustration.
\begin{figure}[h]
\centering
\includegraphics[height=90mm]{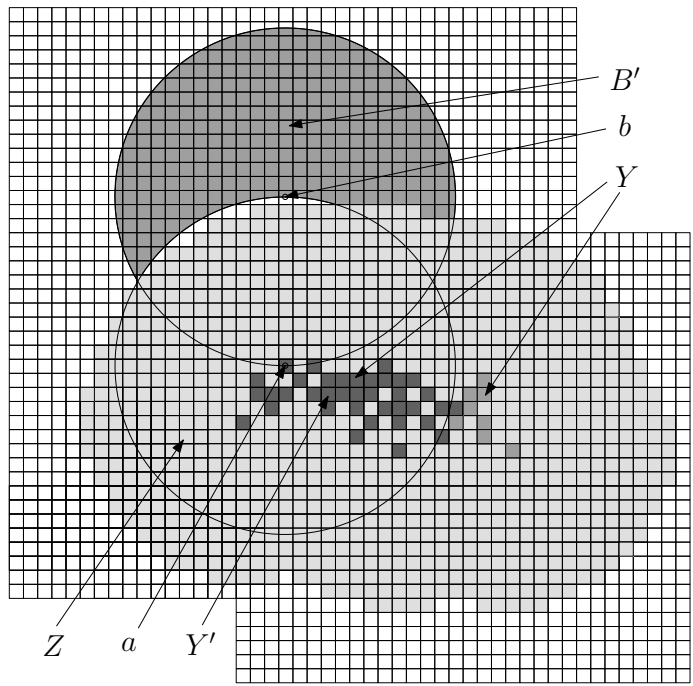}
\caption{The points $a$ and $b$, and the regions $Y$, $Y'$, $Z$ and $B'$.}
\label{FigBasicTile}
\end{figure}
\end{definition}

We can use these new regions to form a analogous version of Lemma~\ref{BBRegions}.

\begin{lem}\label{BasicTiles}
If $G$ contains a component, $X$, of diameter at most $d\sqrt{\log n}$, then with high probability there will be some triple $(a,b,Y)$ such that:
\begin{enumerate}
\item \label{BTDiam}The diameter of $Y$ is at most $d\sqrt{\log n}+2\sqrt{2}s$,
\item \label{BTDist}$b$ is within $d\sqrt{\log n}$ of $a$,
\item \label{BTEmpty}$\#Z=0$, and
\item \label{BTDense}at least one of $\#Y'$ and $\#B'$ is at least $k$.
\end{enumerate}
\end{lem}
\begin{proof}
Given a component $X$, we set $Y$ to be the set of tiles that contain a point in $X$, and $a$ and $b$ to be the pair of points such that $a\in X$, $b\notin X$ that minimise $\rho=\Vert ab\Vert$.
\begin{itemize}
\item Condition~\ref{BTDiam} holds as $\diam(Y)\leq \diam(X)+2\sqrt{s}$.
\item Condition~\ref{BTDist} follows from Lemma~\ref{edgelengths}.
\item Condition~\ref{BTEmpty} follows since no point outside of $X$ can be within $\rho$ of a point in $X$ and every tile of $Y$ contains a point in $X$.
\item Condition~\ref{BTDense} follows since $ab$ is not an edge of $G$, and every location in any tile with its centre within $\rho-\sqrt{2}$ of the centre of a tile containing a point $x\in X$ must be within $\rho$ of $x$.
\end{itemize}
\end{proof}

The Isoperimetric Theorem (Lemma~\ref{IsoLem1}) allows us to bound the area of $Z$ in terms of the area of $Y$:
\begin{lem}\label{IsoLem}
For a triple $(a,b,Y)$, if no tile of $Y$ is within $d\sqrt{\log n}$ of the edge of $S_{n}$ then, writing $r=\rho-\sqrt{2}s>(1-10^{-4})\rho$ (where again $\rho=\Vert ab\Vert$), we have:\[|Z|\geq \pi r^{2}+2r\sqrt{\pi|Y|}\]
If $Y$ does contain a tile within $d\sqrt{\log n}$ of the edge of $S_{n}$, but no tile within $2d\sqrt{\log n}$ of a corner then:\[|Z|\geq \frac{\pi}{2} r^{2}+r\sqrt{\pi|Y|}\]
\end{lem}

\begin{proof}
The Isoperimetric Theorem tells us that the area of $|Z|$ is at least what it would be if $Y$ was a disk and $Z$ was its $r$ blow-up. In this case:
\begin{align*}
\text{radius}(Y) & = \sqrt{|Y|/\pi}
\end{align*}
and so:
\begin{align*}
|Z| & \geq \pi\left(r+\sqrt{\pi/|Y|}\right)^{2}-|Y|\\
& = \pi r^{2}+2r\sqrt{\pi|Y|}
\end{align*}
The second part follows in exactly the same way, using part~2 of our version of the Isoperimetric Theorem.
\end{proof}

With this machinery in place, we can now proceed to prove that as $k$ nears the connectivity threshold, all small components are very small, i.e. of size much less than $k$. The proof works in two parts: We first prove that, with high probability, no triple $(a,b,Y)$ has $\#Y'\geq k$ and $\#Z=0$ for $k\geq 0.7209\log n$. This allows us to conclude that if $G$ contains a small component, then with high probability some triple $(a,b,Y)$ has $B'\geq k$ and $\#Z=0$ by Lemma~\ref{BasicTiles}. We then use this to bound the size of any small component by showing that no triple $(a,b,Y)$ has $\#B'\geq k$, $\#Z=0$ and $\#Y\geq0.309k$ with high probability.

\begin{lem}\label{ANotDense}
If $c>0.7209$ and $k=c\log n$, then with high probability, no triple $(a,b,Y)$ meeting Condition~1-4 of Lemma~\ref{BasicTiles} has $\#Y'\geq k$.
\end{lem}
\begin{proof}
Let $p_{A}(n)$ be the probability that a given triple $(a, b, Y)$ with no part of $Y$ within $d\sqrt{\log n}$ of the boundary of $S_{n}$ and meeting Conditions~\ref{BTDiam} and \ref{BTDist} of Lemma~\ref{BasicTiles} also meets Conditions~\ref{BTEmpty} and has $\#Y'\geq k$. Let $p_{A'}(n)$ be this same probability when $Y$ does contain a tile within $d\sqrt{\log n}$ of the boundary of $S_{n}$.

\begin{Cases}
\item $Y$ does not contain a tile within $d\sqrt{\log n}$ of the boundary of $S_{n}$:

There will be $\Or(n)$ choices for the point $a$, and once $a$ has been chosen, there are only $\Or(\log n)$ choices for $b$ (since it is within $d\sqrt{\log n}$ of $a$), and only a (large) constant number of choices for $Y$, since $Y$ can only include tiles from the fixed collection of $16(dM)^{2}$ tiles nearest to $a$ (i.e. the tiles within $d\sqrt{\log n}$ of $a$). Thus there are $\Or(n\log n)$ possible triples $(a, b, Y)$ meeting Conditions~\ref{BTDiam} and \ref{BTDist} of Lemma~\ref{BasicTiles}.

We show that $p_{A}(n)$ decays at least as fast as $n^{-(1+\varepsilon)}$.

By Lemma~\ref{IsoLem}:
\begin{align}
|Z| & \geq \pi r^{2}+2r\sqrt{\pi|Y|}\notag\\
& \geq \pi r^{2}+2r\sqrt{\pi|Y'|}\notag
\end{align}
where $r=\rho-\sqrt{2}s>(1-10^{-4})\rho$.

Since every tile of $Y'$ contains a location within $\rho+2\sqrt{2}s$ of $a$, and no tile in $Y'$ contains a location within $\rho-2\sqrt{2}s$ of $b$, we have:
\begin{align}
|Y'| & \leq \left(\frac{\pi}{3}+\frac{\sqrt{3}}{2}\right)\rho^{2}+\pi\left((\rho+2\sqrt{2}s)^{2}-\rho^{2}\right)\notag\\
& <\left(\frac{\pi}{3}+\frac{\sqrt{3}}{2}+\frac{\pi}{1000}\right)\rho^{2}\label{AreaYEq}
\end{align}
If $(a,b,Y)$ meets Condition~3 of Lemma~\ref{BasicTiles} (i.e. has $\#Z=0$), and $\#Y'\geq k$, then by Lemma~\ref{Full-Empty}:
\begin{align}
p_{A}(n) & \leq \left(\frac{|Y'|}{|Y'|+|Z|}\right)^{k}\notag\\
& \leq \left(\frac{|Y'|}{\pi r^{2}+2r\sqrt{\pi|Y'|}+|Y'|}\right)^{k}\notag\\
& = \Exp\left(-c\log\left(\frac{\pi r^{2}+2r\sqrt{\pi|Y'|}+|Y'|}{|Y'|}\right)\log n\right)\label{ANotDenseExp}
\end{align}
Maximising (\ref{ANotDenseExp}) over the range $0<|Y'|<\left(\tfrac{\pi}{3}+\tfrac{\sqrt{3}}{2}+\tfrac{\pi}{1000}\right)\rho^{2}$, we achieve a maximum of $n^{-1.18\ldots}$ (when $|Y'|$ is maximal). Thus, with high probability, we will have no system with $\#Y'\geq k$.

\item $Y$ does contain a tile within $d\sqrt{\log n}$ of the boundary of $S_{n}$:

We will have $\Or(n^{1/2})$ choices for $a$, and the same argument as in the previous case shows that there are $\Or(n^{1/2}\log n)$ such triples meeting Conditions~\ref{BTDiam} and \ref{BTDist} of Lemma~\ref{BasicTiles} that also have some tile of $Y$ within $d\sqrt{\log n}$ of the boundary of $S_{n}$.

We show that $p_{A'}(n)$ decays as at least $n^{-(1/2+\varepsilon)}$.

Here Lemma~\ref{IsoLem} only ensures $|Z|\geq\frac{1}{2}\pi r^{2}+r\sqrt{\pi|Y'|}$. Equation (\ref{AreaYEq}) still holds and (\ref{ANotDenseExp}) becomes:
\begin{align}
p'_{A}(n) & \leq \Exp\left(-c\log\left(\frac{\tfrac{1}{2}\pi r^{2}+r\sqrt{\pi|Y'|}+|Y'|}{|Y'|}\right)\log n\right)\label{ANotDenseExp2}
\end{align}
Maximising (\ref{ANotDenseExp2}) over the range $0<|Y'|<\left(\frac{\pi}{3}+\frac{\sqrt{3}}{2}+\frac{\pi}{1000}\right)\rho^{2}$, we achieve a maximum of $n^{-0.81\ldots}$ (again when $|Y'|$ is maximal). Thus again, with high probability, we will have no system with $\#Y'\geq k$, and thus with high probability no small component has $\#Y'\geq k$.
\end{Cases}
\end{proof}

Lemma~\ref{ANotDense} tells us that, with high probability, as $k$ approaches the connectivity threshold, every triple $(a,b,Y)$ that corresponds exactly to a small component, will have $\#B'\geq k$ (i.e. we can change Condition~\ref{BTDense} in Lemma~\ref{BasicTiles} (from $\#A\geq k$ or $\#B'\geq k$) to simply $\#B'\geq k$ (denote this Condition~\ref{BTDense}'), and the Lemma will stay true). We use this to strengthen the previous argument and show that in fact there are far fewer than $k$ points in the whole of any small component, but first need a result about how dense two disjoint regions can be simultaneously. The following is a result about the Poisson process that is a slight alteration of Lemma~6 from [\ref{MW2}] which goes through by exactly the same proof:

\begin{lem}\label{Full-Empty2}
If $X$, $Y$ and $Z$ are three regions with $|X|\leq|Y\cup Z|$, $|Y|\leq |X\cup Z|$ and $X\cap Y=\emptyset$, then, writing $E$ for the event that $\# X\geq mk$, $\#Y\geq k$ and $\#Z=0$, we have:
\begin{align}
\Prb(E) & \leq \left(\frac{2|X|}{|X|+|Y|+|Z|}\right)^{mk}\left(\frac{2|Y|}{|X|+|Y|+|Z|}\right)^{k}
\end{align}
\end{lem}

We can now show, by a similar argument to Lemma~\ref{ANotDense}:

\begin{prop}\label{XSmall}
Let $c>0.7209$ and $k=c\log n$. Then with high probability no small component contains more than $0.309k$ points of $G$.
\end{prop}
\begin{proof}
If $G$ contains a small component with at least $0.309k$ points, then with high probability there will be some triple $(a,b,Y)$ that meets Conditions~\ref{BTDiam}--\ref{BTEmpty} of Lemma~\ref{BasicTiles}, Condition~\ref{BTDense}' and $\#Y\geq 0.309k$. We write $p_{X}$ for the probability that a triple $(a,b,Y)$ meeting Conditions~\ref{BTDiam} and \ref{BTDist} meets the rest of these conditions when $Y$ contains no tile within $d\sqrt{\log n}$ of the boundary of $S_{n}$ and $p_{X'}$ for the same probability when $Y$ does contain such a tile. As in Lemma~\ref{ANotDense} it suffices to show that $p_{X}$ decays at least as fast as $n^{-1-\varepsilon}$ and $p_{X'}$ decays as at least $n^{-1/2-\varepsilon}$ for some $\varepsilon>0$ to complete the proof.

We wish to apply Lemma~\ref{Full-Empty2}, but need to check the conditions of the Lemma first:
\begin{enumerate}
\item The condition $|B'|\leq |Y\cup Z|$ follows as $|Z|\geq \pi r^{2}\approx 3.14\rho^{2}$ and $|B'|\leq(\pi/3+\sqrt{3}/2)\rho^{2}\approx 1.91\rho^{2}$, and so $|Z|\geq|B'|$.
\item The condition that $B'\cap Y=\emptyset$ follows by definition.
\item The condition $|Y|<|B'\cup Z|$: By Lemma~\ref{IsoLem}, $|Z|\geq\pi r^{2}+2 r\sqrt{\pi|Y|}$ when $Y$ contains no tile within $d\sqrt{\log n}$ of the edge of $S_{n}$ and $|Z|\geq\pi r^{2}/2+ r\sqrt{\pi|Y|}$ when $Y$ does. Solving $|Y|>\pi r^{2}+2 r\sqrt{\pi|Y|}$ and $|Y|>\pi r^{2}/2+ r\sqrt{\pi|Y|}$, we gain that $|Y|>11.72\rho^{2}$ and $|Y|>5.861\rho^{2}$ respectively. Thus, so long as $|Y|\leq 11.7\rho^{2}$ in the centre case, and $|Y|\leq 5.86\rho^{2}$ in the edge case, $|Y|<|Z|$, and so the condition holds. When $Y$ exceeds these bounds, we cannot apply Lemma~\ref{Full-Empty2}, but instead note that, for $Y$ in this range:
\begin{align}
p_{X} & \leq \Prb(\#Z=0\text{ and }\#B'\geq k)\notag\\
& \leq \left(\frac{|B'|}{|B'|+|Z|}\right)^{k}\notag\\
& \leq \left(\frac{(\pi /3+\sqrt{3}/2)\rho^{2}}{(\pi /3+\sqrt{3}/2)\rho^{2}+\pi r^{2} + 2r\sqrt{\pi |Y|}}\right)^{k}\notag\\
& < \left( \frac{ \pi /3+\sqrt{3}/2 }{ 4\pi /3 +\sqrt{3}/2 + 2\sqrt{11.7} } \right)^{k}\notag\\
& < n^{-1.58}
\end{align}
By an exact analogy in the edge case, when $|Y|>5.86\rho^{2}$, we find that:
\begin{align}
p_{X'} & < n^{-1.01}
\end{align}

\end{enumerate}

Thus, for $c\geq 0.7209$, and recalling that $r>(1-10^{-4})\rho$:
\begin{align}
p_{X} & \leq \Prb(|Y|\leq 11.7\rho^{2})\Prb\left(\# Z=0,\#B'\geq k,\#Y\geq 0.309k\Big| |Y|\leq 11.7\rho^{2}\right)\notag\\
& \quad + \Prb(|Y|> 11.7\rho^{2})n^{-1.58}\notag\\
& \leq \Max{|Y|\leq11.7\rho^{2}}\left(\frac{2|Y|}{|B'|+|Y|+|Z|}\right)^{0.309k}\left(\frac{2|B'|}{|B'|+|Y|+|Z|}\right)^{k} + n^{-1.58}\notag\\
& \leq \Max{|Y|\leq11.7\rho^{2}}\frac{(2|Y|)^{0.309k}(2(\pi/3+\sqrt{3}/2)\rho^{2})^{k}}{\left((\pi/3+\sqrt{3}/2)\rho^{2}+|Y|+\pi r^{2}+2r\sqrt{\pi|Y|}\right)^{1.309k}} +n^{-1.58}\notag\\
& \leq \Max{|Y|\leq11.7\rho^{2}}\frac{(2|Y|)^{0.309k}(2(\pi/3+\sqrt{3}/2)\rho^{2})^{k}}{\left((\pi/3+\sqrt{3}/2)\rho^{2}+|Y|+\pi r^{2}+2r\sqrt{\pi|Y|}\right)^{1.309k}} +n^{-1.58}\label{XSmallExp}
\end{align}
Maximising the first term over the range $0\leq|Y|\leq11.7\rho^{2}$, we find that the first term of (\ref{XSmallExp}) achieves a maximum of $n^{-1.0001\ldots}$ when $|Y|=0.6069\rho^{2}\ldots$.

Similarly we have:
\begin{align}
p_{X'} & \leq \Prb(|Y|\leq 5.86\rho^{2})\Prb\left(\# Z=0,\#B'\geq k,\#Y\geq 0.309k\Big| |Y|\leq 5.86\rho^{2}\right)\notag\\
& \quad + \Prb(|Y|> 5.86\rho^{2})n^{-1.01}\notag\\
& \leq \Max{|Y|\leq5.86\rho^{2}}\frac{(2|Y|)^{0.309k}(2(\pi/3+\sqrt{3}/2)\rho^{2})^{k}}{\left((\pi/3+\sqrt{3}/2)\rho^{2}+|Y|+\pi r^{2}/2+r\sqrt{\pi|Y|}\right)^{1.309k}}+n^{-1.01}\label{XSmallExp2}
\end{align}
Maximising the first term over the range $0\leq|Y|\leq5.86\rho^{2}$, we find that the first term of (\ref{XSmallExp2}) achieves a maximum of $n^{-0.593\ldots}$ when $|Y|=0.601\rho^{2}$.

Thus, with high probability, no triple $(a,b,Y)$ has $\#Y\geq0.309k$, $\#B'\geq k$ and $\#Z=0$, and so with high probability there is no small component containing more than $0.309k$ points.
\end{proof}

We will use this result to prove a stronger bound on the connectivity threshold. The idea is to show that, with high probability, any triple $(a,b,Y)$ which meets Conditions~\ref{BTDiam}-\ref{BTEmpty} of Lemma~\ref{BasicTiles}, Condition~\ref{BTDense}' and has $\#Y\leq 0.309k$, which we know happens with high probability if $G$ contains a small component, will have another point, $\beta$, in neither $B'$ nor $Y$, but is within $1.0767\rho$ of $a$ such that $\overrightarrow{a\beta}$ is an out edge, but $\overrightarrow{\beta a}$ is not. There must then be a dense region around $\beta$, and we can use this to improve our bound on the connectivity threshold. More precisely we will show that there are $k$ points in the following region:

\begin{definition}
Given the system $(a,b,\beta,Y)$ with $a$, $b$ and $Y$ as usual and $\beta\notin Y\cup B'$, we define the region (shown in Figure~\ref{FigPosBet}):
\[B^{*} = \Bigl[\bigl(D_{\beta}(\Vert a\beta\Vert)\cap B'\bigr)\cup\bigl(D_{\beta}(\Vert a\beta\Vert)\setminus D_{a}(\Vert a\beta\Vert)\bigr)\Bigr]\setminus\bigl( Y\cup Z\bigr)\]
\begin{figure}[h]
\centering
\includegraphics[height=70mm]{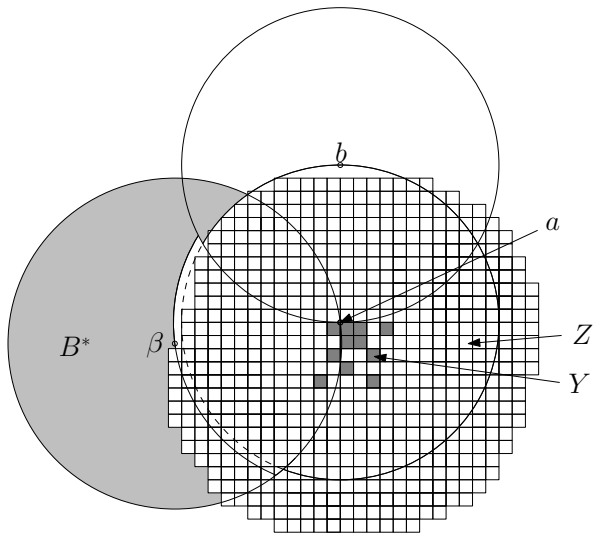}
\caption{The point $\beta$ and the region $B^{*}$.}
\label{FigPosBet}
\end{figure}
\end{definition}

We introduce one more piece of notation, and then prove that there will be a suitable $\beta$ with high probability.

\begin{definition}
Given $\lambda>\rho$, we write $B(\lambda)=B'\cap D_{a}(\lambda)$ and $A(\lambda)=D_{a}(\lambda)\setminus\left(D_{a}(\rho)\cup B\right)$. See Figure~\ref{FigAlandBl}.
\end{definition}
\begin{figure}[h]
\centering
\includegraphics[height=70mm]{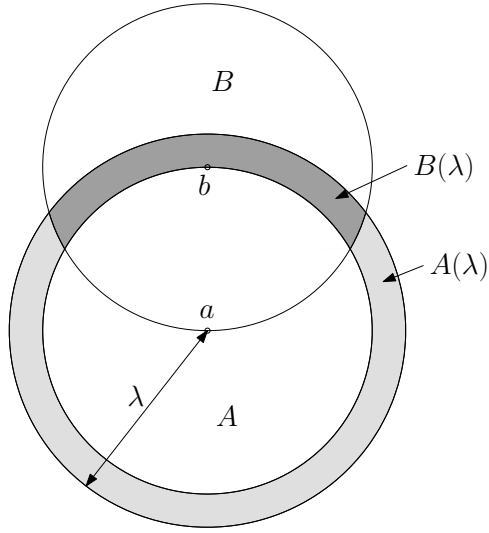}
\caption{The region $A(\lambda)$ and $B(\lambda)$.}
\label{FigAlandBl}
\end{figure}

The following lemma tells us that with high probability, if $G$ contains a small component, then we can find a suitable point $\beta$.

\begin{lem}\label{TBound1}
If $k>0.9684\log n$ and $G$ contains a component of diameter at most $d\sqrt{\log n}$, then with high probability there is some quadruple $(a,b,\beta,Y)$ such that:
\begin{enumerate}
\item The diameter of $Y$ is at most $d\sqrt{\log n}+2\sqrt{2}s$,\label{BTDiam2}
\item $b$ is within $d\sqrt{\log n}$ of $a$,
\item $\#Z=0$,\label{BTZEmpty2}
\item \label{BTDense2}$\#B'\geq k$,
\item \label{BTDense3}$\#Y\leq0.309k$,
\item \label{BTNoBoundary}$Y$ contains no tile within $d\sqrt{\log n}$ of the boundary of $S_{n}$,
\item \label{BTBeta}$\beta\in A(1.0767\rho)$ and
\item \label{BTBDense}$\#B^{*}\geq k$.
\end{enumerate}
\end{lem}
\begin{proof}
Given a small component, $X$, we take $Y$ to be exactly the tiles that meet $X$ and $a$ and $b$ to be the pair such that $a\in X$, $b\notin X$ and $\Vert ab\Vert$ is minimal, all as usual. Then Conditions~\ref{BTDiam2}--\ref{BTZEmpty2} are met with high probability by Lemma~\ref{BasicTiles}, Condition~\ref{BTDense2} is met by Lemma~\ref{ANotDense}, Condition~\ref{BTDense3} is met by Proposition~\ref{XSmall} and Condition~\ref{BTNoBoundary} is met by Lemma~\ref{NoBoundaries}. We take $\beta$ to be the point outside of $B'\cup Y\cup Z$ that is closest to $a$.

To show Condition~\ref{BTBeta} holds with high probability we show that no triple $(a,b,Y)$ meeting Conditions~\ref{BTDiam} and \ref{BTDist} has both:
\begin{enumerate}
\item $\#B'\geq k$ and,
\item $\#\Bigl(Z\cup A(1.0767\rho)\setminus Y\Bigr)=0$.
\end{enumerate}
If, with high probability, this does not occur, then with high probability there will be some point in $A(1.0767\rho)$, and so in particular $\beta\in A(1.0767\rho)$.

We write $E_{1}$ for the event that a particular triple has $\#B'\geq k$, $\#\bigl(Z\cup A(1.0767\rho)\setminus Y\bigr)=0$ and meets Conditions~\ref{BTDiam} and \ref{BTDist}. We know that $|B'|\leq \pi/3+\sqrt{3}/2$ and, by Lemma~\ref{Full-Empty}:\[\Prb(E_{1})\leq \left(\frac{|B'|}{|B'|+|Z\cup A(1.0767\rho)\setminus Y|}\right)^{k}\]
Thus $\Prb(E_{1})$ will be maximised when $B'$ is maximised and $|\bigl(A(1.0767\rho)\cup Z\bigr)\setminus Y|$ is minimised. By the Isoperimetric Theorem, this will occur when $Y$ is the small disk centred on $a$ whose $r$ blow-up just covers $A(1.0767\rho)$. In this case:
\[\text{radius}(Y)=1.0767\rho-r\leq 0.0768\rho\]
And so, omitting the trivial but tedious calculations to evaluate $|A(1.0767\rho)|$:
\begin{align}
|\bigl(Z\cup A(1.0767\rho)\setminus Y\bigr)| & \geq |D_{a}(\rho)|+|A(1.0767\rho)|-\pi (0.0768\rho)^{2}\notag\\
& > 3.4602\rho^{2}\notag\\
\end{align}
Thus:
\begin{align}
\Prb(E_{1}) & \leq \left(\frac{|B'|}{|B'|+|Z\cup A(1.0767\rho)\setminus Y|}\right)^{k}\notag\\
&\leq \left(\frac{(\pi/3+\sqrt{3}/2)\rho^{2}}{(\pi/3+\sqrt{3}/2)\rho^{2}+3.4602\rho^{2}}\right)^{0.9684\log n}\notag\\
&< n^{-1.00004}\label{ShowingBeta1}
\end{align}
Since there are only $\Or(n\log n)$ such systems, (\ref{ShowingBeta1}) tells us that $E_{1}$ will not occur for any of them with high probability, and so Condition~\ref{BTBeta} holds with high probability.

To show Condition~\ref{BTBDense} holds with high probability we first show that $\overrightarrow{a\beta}$ is an out edge with high probability. Then since $a\beta$ cannot be an edge, $D_{\beta}(\Vert a\beta\Vert)$ must contain $k$ points, and we finish the proof by showing that the nearest $k$ of these to $\beta$ will all lie in $B^{*}$ with high probability.

If some small component did not have $\overrightarrow{a\beta}$ being an out-edge, then, since $\#Y\leq 0.309k$, there would be at least $(1-0.309)k=0.691k$ points in $B(\Vert a\beta\Vert)\subset B(1.0767\rho)$. Then there would be some triple $(a,b,Y)$ with $\# B(1.0767\rho)\geq 0.691k$ and $\# Z=0$. We write $E_{2}$ for the event that a given triple meeting Conditions~\ref{BTDiam} and \ref{BTDist} has $\# B(1.0767\rho)\geq 0.691k$ and $\# Z=0$. Calculations show that $|B(1.0767\rho)|\leq 0.1632\rho^{2}$, and we know that $|Z|\geq \pi r^{2}$, thus:
\begin{align}
\Prb(E_{2}) & \leq \left(\frac{|B(1.0767)|}{||B(1.0767\rho)\cup Z|}\right)^{0.691k}\notag\\
& \leq \left(\frac{0.1632\rho^{2}}{0.1632\rho^{2}+\pi r^{2}}\right)\notag\\
& < n^{-2.3}
\end{align}
Thus, $E_{2}$ does not occur for any triple $(a,b,Y)$ with high probability, and so $\overrightarrow{a\beta}$ will be an out edge with high probability.

This tells us that $D_{\beta}(\Vert a\beta\Vert)$ must contain $k$ points, and we know that none of these points are in $Z\cup A(\Vert a\beta\Vert)$. Thus they must lie in $B^{*}\cup Y$. We complete the proof by showing that with high probability none of the $k$-nearest neighbours of $\beta$ lie in $Y$.

If there were a point, $\gamma$, in $D_{\beta}(\Vert a\beta\Vert)\cap Y$ such that $\gamma$ was one of the $k$-nearest neighbours of $\beta$, then there must be $k$ points within $D_{\gamma}(\Vert\beta\gamma\Vert)$ since $\beta\gamma$ is not an edge of $G$. At most $0.309k$ of these can be in $Y$ by Proposition~\ref{XSmall}, and no other points can be within $D_{\gamma}(\rho)$. Thus there must be at least $0.691k$ points within $D_{\gamma}(\Vert\beta\gamma\Vert)\setminus (D_{\gamma}(\rho)\cup Y\cup Z)\subset D_{\gamma}(\Vert\beta\gamma\Vert)\setminus (D_{\gamma}(\rho)\cup Z)$. 

Given a system $(a,b,\beta,\gamma,Y)$ with $a$, $b$, and $Y$ as before, $\beta\in A(1.0767\rho)$ and $\gamma\in D_{\beta}(\Vert a\beta\Vert)\cap Y$, we write $E_{3}$ for the event that $\# Z=0$ and $\# D_{\gamma}(\Vert\beta\gamma\Vert)\setminus (D_{\gamma}(\rho)\cup Z)\geq 0.691k$. We know $|Z|\geq \pi r^{2}$ and $|D_{\gamma}(\Vert\beta\gamma\Vert)\setminus (D_{\gamma}(\rho)\cup Z)|\leq \pi (1.0767^{2}-1)\rho^{2}$, thus:
\begin{align}
\Prb(E_{3}) & \leq \left(\frac{|D_{\gamma}(\Vert\beta\gamma\Vert)\setminus (D_{\gamma}(\rho)\cup Z)|}{|Z\cup D_{\gamma}(\Vert\beta\gamma\Vert)\setminus (D_{\gamma}(\rho)\cup Z)|}\right)^{0.691k}\notag\\
& \leq \left(\frac{\pi (1.0767^{2}-1)\rho^{2}}{\pi (r^{2}+1.0767^{2}\rho^{2}-\rho^{2})}\right)^{0.691k}\notag\\
& < n^{-1.3}
\end{align}
Thus, with high probability, $E_{3}$ does not occur for any such system $(a,b,\beta,\gamma,Y)$, and so in particular none of the $k$ nearest neighbours of $\beta$ will be in $Y$ with high probability, and so we will have $\# B^{*}\geq k$ with high probability as required.
\end{proof}

We can now prove our stronger bound on the connectivity threshold, but first state a result about the probability of two intersecting regions being dense, which can be read out of the proof of Theorem~15 of [\ref{MW}].

\begin{lem}\label{IntersectingLemma}
Let $A_{1}$, $A_{2}$, $A_{3}$ and $A_{4}$ be four disjoint regions of $S_{n}$ and let $n_{i}=\# A_{i}$. Then, so long as $|A_{1}|\leq|A_{3}|<2|A_{1}|$, we have:
\begin{align*}
\Prb(n_{1}+n_{2}\geq k\textrm{, }n_{2}+n_{3}\geq k\textrm{ and }n_{4}=0) & \leq \mu^{-k}n^{o(1)}
\end{align*}
where $\mu$ is the solution to:
\begin{align*}
\sum_{i=1}^{4}|A_{i}|=\mu|A_{2}|+\sqrt{4\mu|A_{1}||A_{3}|}
\end{align*}
\flushright{$\square$}
\end{lem}

\begin{thm-hand}{\ref{TightBoundThm}}
If $k=c\log n$ and $c>0.9684$, then $G$ is connected with high probability.
\end{thm-hand}
\begin{proof}
We know that if $G$ contains a small component then with high probability there will be a system $(a,b,\beta,Y)$ meeting all the conditions of Lemma~\ref{TBound1}. We show that for $c>0.9684$ no such system meets all these conditions with high probability.

Given a system $(a,b,\beta,Y)$ meeting Conditions~\ref{BTDiam}, \ref{BTDist}, \ref{BTNoBoundary} and \ref{BTBeta} of Lemma~\ref{TBound1} (so that there are $\Or(n(\log n)^{2})$ such systems), we write $E$ for the event $\#B'\geq k$ and $\#B^{*}\geq k$ and set:
\begin{align*}
B_{1} & = B'\setminus B^{*}\\
B_{2} & = B'\cap B^{*}\\
B_{3} & = B^{*}\setminus B'
\end{align*}
We write $n_{i}=\# B_{i}$ for $(i=1,2,3)$, $n_{4}=\#Z$, then $E$ is the event $n_{1}+n_{2}\geq k$, $n_{2}+n_{3}\geq k$ and $n_{4}=0$.

We wish to apply Lemma~\ref{IntersectingLemma}, but need to make sure that either $|B_{1}|\leq|B_{3}|< 2|B_{1}|$ or $|B_{3}|\leq|B_{1}|< 2|B_{3}|$. We know that $|B'|\leq (\tfrac{\pi}{3}+\tfrac{\sqrt{3}}{2})\rho^{2}$ and calculations show that $|B^{*}|<2.31\rho^{2}$ and $|B'\cap B^{*}|<0.6515\rho^{2}$. From this it is easily checked that the conditions will hold unless at least one of $|B^{*}|$ or $|B'|$ is small whilst the other is large, in particular, at least one of $|B_{1}|\leq|B_{3}|< 2|B_{1}|$ or $|B_{3}|\leq|B_{1}|< 2|B_{3}|$ will hold so long as $|B^{*}|\geq 1.73\rho^{2}$ and $|B'|\geq 1.73\rho^{2}$. When one of these does not hold, we note that:\[\Prb(E)\leq\Prb(\#Z=0\text{ and }\#B'=0)\]And:\[\Prb(E)\leq\Prb(\#Z=0\text{ and }\#B^{*}=0)\]And apply Lemma~\ref{Full-Empty}. Thus we have:
\begin{align}
\Prb(E) & \leq \Prb(|B'|,|B^{*}|\geq1.73\rho^{2})\Prb(E\big| |B'|,|B^{*}|\geq1.73\rho^{2})\notag\\
& \quad +\Prb(|B'|<1.73)\Prb(E\big| |B'|<1.73)\notag\\
& \quad +\Prb(|B^{*}|<1.73)\Prb(E\big| |B^{*}|<1.73)\notag\\
& \leq \Max{|B'|,|B^{*}|\geq1.73\rho^{2}}\mu^{-k}n^{o(1)} + \Max{|B'|<1.73\rho^{2}}\left(\frac{|B'|}{|B'|+|Z|}\right)^{k}\notag\\
& \quad + \Max{|B^{*}|<1.73\rho^{2}}\left(\frac{|B^{*}|}{|B^{*}|+|Z|}\right)^{k}\notag\\
& < \Max{|B'|,|B^{*}|\geq1.73\rho^{2}}\mu^{-k}n^{o(1)} + 2\left(\frac{1.73\rho^{2}}{1.73\rho^{2}+\pi r^{2}}\right)^{k}\notag\\
& \leq \Max{|B'|,|B^{*}|\geq1.73\rho^{2}}\mu^{-k}n^{o(1)} + 2n^{-1.01}\label{ThmExp2}
\end{align}
where:
\begin{align}
|Z|+\sum_{i} |B_{i}|= \mu|B_{2}| + \sqrt{4\mu|B_{1}||B_{3}|}\label{ThmEqn3}
\end{align}
Thus $\Prb(E)$ will be maximised exactly when $\mu$ is minimised, which will be when $B^{*}$ overlaps with $B'$ as much as possible and $|B'|$ and $|B^{*}|$ are maximal. This will happen when $\beta$ is located at $\partial D_{a}(1.0767\rho)\cap \partial B'$. Calculating $\mu$ in this case yields $\mu>2.8087$.

Using this, we gain that the exponent of the first term of (\ref{ThmExp2}) is strictly less than $-1$ for $c>0.9684$, and so if $c>0.9684$, $E$ will not occur for any system $(a,b,\beta,Y)$ with high probability, and so, with high probability, $G$ will be connected.
\end{proof}

\section{Conclusion and Open Questions}
In the last section we worked quite hard to bring the bound for the connectivity threshold down below $\log n$. However, the bound we proved, $0.9684\log n$, is actually lower than the previously best known bound for the directed model of $0.9967\log n$ proved in [\ref{MW2}], and so since the edge in our strict undirected model are exactly the bidirectional edges in the connected model, it improves the bound for the directed model as well.

In fact, we believe a much stronger result holds. It seems that in both the directed model and strict undirected model the barrier to connectivity is an isolated vertex (or at least a very concentrated cluster of sub-logarithmic size). If this is the case, then it seems likely that the connectivity threshold for both models is the same (this does not immediately follow from the barrier in both cases being an isolated vertex, since in the directed model the isolated vertex is in an in-component by itself, where as it may be possible that an isolated point in the strict undirected model has in-edges, but not from any of its $k$-nearest neighbours, however set-ups where this occurs seem less likely than an isolated vertex in an in-component).

In fact, the lower bound proved on the connectivity threshold for both models is essentially the threshold for having a point with no in-edges, and so putting this all together motivates the following conjecture:

\begin{conjecture}
The barrier for connectivity for both the directed model and the strict undirected model, is an isolated vertex (or concentrated cluster of sub-logarithmic size) with no in-edges, and so the connectivity threshold in both models is the same (and something a little over $0.7209\log n$).
\end{conjecture}

It is possible to strengthen the bounds of several of the results proved in this paper (although with a fair amount of extra work). The upper bound on the size of a small component around the connectivity threshold of $0.309\log n$ (Lemma~\ref{XSmall}) can be improved to $0.203\log n$ by using a stronger version of Lemma~\ref{Full-Empty2} (although the conditions needed to apply it then require more work to check).

The bound on the threshold for the edges of different components crossing (Theorem~\ref{nocrossing}) can also be improved significantly. By determining the exact positions of $a_{1}$ and $a_{2}$ that maximise the ratio $|H|/|H\cup L|$ the bound can be reduced to around $0.5\log n$, although this is almost certainly still a long way off the actual threshold.




\begin{appendix}
\section{Definitions and Notation from Section \ref{NoCrossSection}}\label{DefApp}
We collate here all the definitions and notation used in Section \ref{NoCrossSection} in the order in which they appear.

\begin{itemize}
\item We say that $a_{1}$, $a_{2}$, $b_{1}$ and $b_{2}$ form a \emph{crossing pair} if there are two different components $X$ and $Y$ with $a_{1}$, $a_{2}\in X$, $b_{1}$, $b_{2}\in Y$ and the straight line segments $a_{1}a_{2}$ and $b_{1}b_{2}$ intersect and are both in the graph $G$, such that $\Vert a_{1}a_{2}\Vert\leq\Vert b_{1}b_{2}\Vert$, $\Vert a_{1}b_{1}\Vert\leq \Vert a_{1}b_{2}\Vert$ and $\textrm{d}(a_{1},b_{1}b_{2})\leq\textrm{d}(a_{2},b_{1}b_{2})$. 
\item For $i=1,2$, $r_{i}=\min\{\Vert a_{i}b_{1}\Vert,\Vert a_{i}b_{2}\Vert\}$ (so that $r_{1}=\Vert a_{1}b_{1}\Vert$.
\item For $i=1,2$, $A_{i}=D_{a_{i}}(r_{i})$.
\item For $i=1,2$, $B_{i}=D_{b_{i}}(1)$.
\item $w=(\frac{1}{2},\frac{1}{2\sqrt{3}})$.
\item $T$ is the triangle with vertices $b_{1}$, $b_{2}$ and $w$.
\item $S_{1}$ is the region $T\setminus(D_{b_{1}}(\frac{1}{2})\cup D_{b_{2}}(\frac{1}{2}))$.
\item $z=(\frac{1}{2},-\frac{\sqrt{3}}{2})$.
\item $T_{2}$ is the triangle with vertices $b_{1}$, $b_{2}$ and $z$.
\item $S_{2}$ is the region $T_{2}\cap A_{1}\cap \{x\in S_{n}:x\widehat{b_{1}}b_{2}>\frac{\pi}{6}\textrm{ and }x\widehat{b_{2}}b_{1}>\frac{\pi}{6}\}$.
\item $R_{1}$ is the region $D^{k}(a_{1})\cap(B_{1}\setminus B_{2})$ and $R_{2}$ is the region $D^{k}(a_{1})\cap(B_{2}\setminus B_{1})$.
\item For $i=1,2$, $E_{i}$ is the elliptical region $\{x\in S_{n}:\Vert b_{i}x\Vert+\Vert a_{1}x\Vert\leq 1$. We write $E_{i}(a_{1})$ for this ellipse when $a_{1}$ is specified.
\item For $i=1,2$, $F_{i}$ is the elliptical region $\{x\in S_{n}:\Vert b_{i}x\Vert+\Vert a_{2}x\Vert\leq 1$. We write $F_{i}(a_{1})$ for this ellipse when $a_{2}$ is specified.
\item For a set $S\subset S_{n}$, we write $S^{+}$ for the part of $S$ which lies above the line through $b_{1}$ and $b_{2}$, and $S^{-}$ for the part of $S$ which lies below the line $b_{1}$ and $b_{2}$.
\item $M$ for the region $D^{k}(a_{1})\cap D^{k}(a_{2})$.
\item $L_{1}=(D^{k}(a_{1})\cap E_{1}\cap D_{b_{1}}(1/2))\setminus M$.
\item $L_{2}=(D^{k}(a_{1})\cap E_{2}\cap D_{b_{2}}(1/2))\setminus M$.
\item $L_{3}=M^{+}\cap D_{b_{1}}(1/2)\cap D_{b_{2}}(1/2)$.
\item $L_{4}=T_{2}\cap D^{k}(a_{2})\cap \{x:x\widehat{b_{1}}b_{2}\leq \pi/6\textrm{ or }x\widehat{b_{2}}b_{1}\leq \pi/6\}$.
\item $L_{5}=(D^{k}(a_{2})\cap F_{1}\cap D_{b_{1}}(1/2))\setminus T_{2}$.
\item $L_{6}=(D^{k}(a_{2})\cap F_{2}\cap D_{b_{2}}(1/2))\setminus T_{2}$.
\item $H_{1}=R_{1}\setminus L{1}$.
\item $H_{2}=R_{2}\setminus L{2}$.
\item $H_{3}=A_{2}\setminus (B_{1}\cup B_{2})$.
\item $H_{4}=M^{+}\setminus L_{3}$.
\item $H = S_{2}\cup\bigcup_{i=1}^{4}H_{i}$.
\item $L = \bigcup_{i=1}^{6}L_{i}$.
\item $v^{+}=(\frac{3}{4},\frac{\sqrt{3}}{4})$.
\item $v^{-}=(\frac{3}{4},-\frac{\sqrt{3}}{4})$.
\item $u^{+}=(\frac{1}{4},\frac{\sqrt{3}}{4})$.
\item $u^{-}=(\frac{1}{4},-\frac{\sqrt{3}}{4})$.
\item $w'=(\frac{1}{2},-\frac{1}{2\sqrt{3}})$.
\item For $i=1,2$, $\rho_{i}$ is the radius of $D^{k}(a_{i})$.
\end{itemize}

\end{appendix}


%
%
%
%


\begin{thebibliography}{99}
\footnotesize
\bibitem{key-1}\label{MW}P. Balister, B. Bollob\'{a}s, A. Sarkar and M.Walters, \emph{Connectivity of random $k$-nearest neighbour graphs}, Advances in Applied Probability, \textbf{37}(1):1--24 (2005)
\bibitem{key-2}\label{ConectConst}P. Balister, B. Bollob\'{a}s, A. Sarkar and M.Walters, \emph{A critical constant for the $k$-nearest neighbour model}, Advances in Applied Probability, \textbf{41}(1):1--12 (2009)
\bibitem{key-3}\label{BMI}R. J. Gardner, \emph{The Brunn-Minkowski inequality}, Bull. Amer. Math. Soc. (NS), \textbf{39} (3) (2002), 355--405
\bibitem{key-4}\label{Gil}E. N. Gilbert, \emph{Random Plane Networks}, Journal of the Society for Industrial Applied Mathematics \textbf{9} (1961), 533--543.
\bibitem{key-5}\label{Pen}M.D. Penrose, \emph{The longest edge of the random minimal spanning tree}, Annals of Applied Probability \textbf{7} (1997), 340--361.
\bibitem{key-6}\label{MW2}M. Walters, \emph{Small components in $k$-nearest neighbour graphs}. Preprint.
\bibitem{key-7}\label{XandK} F. Xue and P. R. Kumar, \emph{The number of neighbors needed for connectivity of wireless networks.} Wireless Networks \textbf{10} (2004), 169--181
\end{thebibliography}
\end{document}